\documentclass[paper=a4,headsepline,DIV13,captions=tableabove]{scrartcl}

\usepackage[utf8]{inputenc}
\usepackage[T1]{fontenc}
\usepackage{lmodern}
\usepackage{amsmath,amsfonts,amssymb,amsthm,mathtools}
\usepackage{graphicx,subfigure}
\usepackage{braket,dsfont,url,enumerate}
\usepackage{hyperref}
\usepackage{booktabs,multirow}
\usepackage{pgfplots}

\renewenvironment{abstract}{\minisec{Abstract}}{\par\vspace{.1in}}
\newenvironment{keywords}{\minisec{Key Words}}{\par\vspace{.1in}}
\newenvironment{AMS}{\minisec{AMS subject classification}}{\par\vspace{.1in}}

\theoremstyle{plain}
\newtheorem{theorem}{Theorem}[section]

\newtheorem{lemma}[theorem]{Lemma}
\newtheorem{prop}[theorem]{Proposition}

\theoremstyle{remark}
\newtheorem{remark}[theorem]{Remark}

\theoremstyle{definition}
\newtheorem{assumption}{Assumption}

\numberwithin{equation}{section}

\newcommand{\R}{\mathds{R}}
\newcommand{\N}{\mathds{N}}

\newcommand{\Om}{\Omega}

\newcommand{\norm}[1]{\lVert#1\rVert}
\newcommand{\abs}[1]{\lvert#1\rvert}
\newcommand{\PP}{\mathcal{P}}
\newcommand{\lh}{\abs{\ln{h}}}

\newcommand{\Qad}{Q_{\text{ad}}}
\newcommand{\Qhad}{Q_{h,\text{ad}}}
\newcommand{\Th}{\mathcal{T}_h}

\DeclareMathOperator{\sgn}{sgn}
\renewcommand{\phi}{\varphi}

\begin{document}

\title{Finite Element Error Estimates for Optimal Control Problems with Pointwise Tracking}

\author{
  Niklas Behringer\footnotemark[2] \and
  Dominik Meidner\footnotemark[2] \and Boris Vexler\footnotemark[2]
}

\markright{Behringer, Meidner, Vexler: FEM Error Estimates for Problems with Pointwise Tracking}

\maketitle

\renewcommand{\thefootnote}{\fnsymbol{footnote}}

\footnotetext[2]{Technical University of Munich, Department of Mathematics, Chair of Optimal
Control, Garching / Germany (nbehring@ma.tum.de, meidner@ma.tum.de, vexler@ma.tum.de)}

\renewcommand{\thefootnote}{\arabic{footnote}}

\begin{abstract}
  We consider a linear-quadratic elliptic optimal control problem with point evaluations of the
  state variable in the cost functional. The state variable is discretized by conforming linear
  finite elements. For control discretization, three different approaches are considered. The main
  goal of the paper is to significantly improve known a priori discretization error estimates for
  this problem. We prove optimal error estimates for cellwise constant control discretizations in
  two and three space dimensions.  Further, in two space dimensions, optimal error estimates for
  variational discretization and for the post-processing approach are derived.
\end{abstract}



%%%%%%%%%%%%%%%%%%%%%%%%%%%%%%%%%%%%%%%%%%%%%%%%%%%%%%%%%%%%%%%%%%%%%%%%%%%%%%%

%%%%%%%%%%%%%%%%%%%%%%%%%%%%%%%%%%%%%%%%%%%%%%%%%%%%%%%%%%%%%%
\section{Introduction}
%%%%%%%%%%%%%%%%%%%%%%%%%%%%%%%%%%%%%%%%%%%%%%%%%%%%%%%%%%%%%%

In this article, we develop a priori error estimates for the discretization of a linear-quadratic
elliptic optimal control problem with point evaluation of the state variable in the cost functional.
That is, we consider the following problem:
\begin{subequations}\label{definition:contproblem1}
  \begin{equation}\label{eq:functional}
\text{Minimize }\;    \frac{1}{2}\sum_{i\in I} (u(x_i)-\xi_i)^2 + \frac{\alpha}{2}\norm{q}^2_{L^2(\Omega)}
  \end{equation}
  subject to
  \begin{equation}\label{eq:equation}
    \begin{aligned}
      -\Delta u &= q &&\qquad\text{in } \Omega\\
      u &=0 &&\qquad\text{on } \partial \Omega,
    \end{aligned}
  \end{equation}
  and
  \begin{equation}\label{eq:constraints}
    a\le q(x)\le b\quad\text{for a.a }x\in\Omega.
  \end{equation}
\end{subequations}
Here, $\Omega \subset \R^{d}$ is a bounded, convex, polygonal/polyhedral domain with
$d\in\set{2,3}$, $\alpha >0$, $a,b\in\R$, and $a<b$.  Problem \eqref{definition:contproblem1} seeks
to minimize the distance of the state $u$ to prescribed values $\xi_i\in\R$ at fixed points $x_i$ in
the interior of $\Omega$ for $i\in I=\set{1,2,\dots,N}$.

The consideration of such a cost functional involving pointwise evaluation is motivated by parameter
identification problems with pointwise measurements, see, e.g.,~\cite{MR2193509} for
numerical analysis of a problem with a similar cost functional and finite dimensional control
(parameter) variable. The optimal control problem~\eqref{definition:contproblem1} and its finite
element discretization is considered and analyzed in recent
publications~\cite{MR3449612,MR3523574,AnOtSa}. The goal of this paper is to significantly improve
the a priori error estimates from these papers, see Table~\ref{tab:overview} and the detailed
discussion below.

\begin{table}[htb]
  \centering
  \caption{Comparison of the orders of convergence of $\norm{\bar q- \bar q_h}_{L^2(\Om)}$} \label{tab:overview}
  \begin{tabular}{llcccc}
    \toprule
    Discretization && variational & cellwise constant & post processing\\
    \midrule
    \multirow{2}{*}{Known results} & $d=2$& $h$ &  $h\abs{\ln h}$ \\
                                   & $d=3$& $C(\varepsilon)h^{1-\varepsilon}$ &  $h^{\frac{1}{2}}\abs{\ln h}^2$ \\
    \midrule
    \multirow{2}{*}{Our approach} & $d=2$& $h^2 \abs{\ln h}^2$ & $h \abs{\ln h}$ & $h^2 \abs{\ln h}^2$  \\ 
                                  & $d=3$& $h \abs{\ln h}$ &  $h \abs{\ln h}$\\ 
    \bottomrule
  \end{tabular}
\end{table}

As in the above mentioned papers, we discretize the state equation~\eqref{eq:equation} with
conforming linear finite elements, see Section~\ref{sec:disc} for details. With $h>0$, we denote the
discretization parameter describing the maximal mesh size. For the control discretization, we
consider the following three approaches. For each of the approaches error estimates for the error
between the optimal control $\bar q$ and the discrete optimal control $\bar q_h$ in terms of the
discretization parameter $h$ are derived.

\begin{itemize}
  \item\emph{Variational discretization:} In this case, the control variable is not discretized
    explicitly. It is implicitly discretized through the optimality system leading to an optimal
    discrete control being not a mesh function, see~\cite{MR2122182}. For this case, estimates of
    order $\mathcal{O}(h)$  for $d=2$ and $\mathcal{O}(h^{1-\varepsilon})$ for $d=3$ are provided
    in~\cite{MR3449612,AnOtSa}. For the two-dimensional case we prove a quasi-optimal (up to a
    logarithmic term) estimate of order $\mathcal{O}(h^2\abs{\ln h}^2)$, see
    Theorem~\ref{theorem:result1.1}.
  \item\emph{Cellwise constant discretization:} In this case the, control variable is discretized
    by cellwise constant functions on the same mesh as used for the state variable. For this choice,
    error estimates of order $\mathcal{O}(h\abs{\ln h})$ and $\mathcal{O}(h^{\frac{1}{2}}\abs{\ln
    h}^2)$ for the dimensions $d=2$ and $d=3$, respectively, are derived in~\cite{AnOtSa}. We prove
    an error estimate $\mathcal{O}(h\abs{\ln h})$ for both $d=2$ and $d=3$ significantly improving
    the known result for the three-dimensional case, see Theorem~\ref{theorem:result2}.
  \item\emph{Post processing approach:} In this case, we use the approach suggested
    in~\cite{MR2114385} for an optimal control problem with $L^2$ tracking. That is, we discretize
    the control variable by cellwise constant functions and define a post-processed control through
    a projection formula~\eqref{eq:q_hat}. To our best knowledge, there are no results for this
    approach in the context of pointwise tracking in the literature. For the two-dimensional case we
    prove the same rate of convergence as for the variational discretization, i.e.,
    $\mathcal{O}(h^2\abs{\ln h}^2)$, see Theorem~\ref{theorem:result5}.
\end{itemize} 

These results in the case of cellwise constant discretizations for both $d=2$ and $d=3$ as well as
in the cases of variational discretization and of the post-processing approach for $d=2$ can not be
further improved, see Table~\ref{tab:overview}. To our best knowledge, the question is open, if it
is possible to prove second order error estimates (up to logarithmic terms) for the
three-dimensional case ($d=3$) on general quasi-uniform meshes. A possible way out is to use graded
meshes locally refined towards the points $\set{x_i}$. By the techniques from~\cite{MR3071172}, it
seems to be directly possible to prove second order estimates on such meshes. On such meshes also
the case of absence of one or both of the control bounds (i.e. $a=-\infty$ or/and $b=\infty$) can be
covered.

The structure of the paper is as follows. In the next section, we discuss the functional-analytic
setting of the problem, provide optimality conditions and derive regularity results. It turns out,
that although the adjoint state $\bar z$ possesses in general only $W^{1,s}(\Om)$ regularity with $s
< d/(d-1)$, the optimal control $\bar q$ is Lipschitz continuous due the presence of control
constrains. In Section~\ref{sec:disc}, the discrete problem for different control discretizations is
introduced and the corresponding optimality conditions are stated. After some estimates for an
auxiliary equation in Section~\ref{sec:aux}, we discuss in Section~\ref{sec:green} properties of the
continuous and the discrete Green's functions. Thanks to recent results from~\cite{MR3614014}, we
are able to show that the discrete Green's function has similar growth behavior close to the
singularity as the continuous one. This property is an important ingredient to prove our main
results but is also of an independent interest.  Then, we prove error estimates of order ${\mathcal
O}(h\abs{\ln h})$ (see Table~\ref{tab:overview}) in Section~\ref{sec:error} and finally estimates
of order ${\mathcal O}(h^2\abs{\ln h}^2)$ in Section~\ref{sec:imp2}. The last section is devoted
to numerical results illustrating our error estimates. For both, the cellwise constant
discretization and the post-processing approach in two dimensions, the numerical results are
fully in agreement with the presented theory. We present also a three-dimensional example for the
post-processing approach and observe second order convergence, which is not covered by our theory,
see the discussion above and in Section~\ref{sec:num} of this issue.

%%%%%%%%%%%%%%%%%%%%%%%%%%%%%%%%%%%%%%%%%%%%%%%%%%%%%%%%%%%%%%
\section{Continuous Problem}\label{sec:continuousproblem}
%%%%%%%%%%%%%%%%%%%%%%%%%%%%%%%%%%%%%%%%%%%%%%%%%%%%%%%%%%%%%%

In this section, we give a rigorous definition of the continuous optimal control problem
\eqref{definition:contproblem1} and derive an optimality system as well as first regularity results
for the optimal control. To this end, let $Q=L^2(\Om)$ and
\[
  \Qad=\Set{q\in Q | a\leq q \leq b \text{ a.e.\ in } \Om}.
\]
Further, let $\frac{2d}{d+2}<s<\frac{d}{d-1}$ and $\frac{1}{s}+\frac{1}{s'}=1$. Then, in particular
there holds $s'>d$, $W^{1,s}(\Omega) \hookrightarrow L^2(\Omega)$, and $W^{1,s'}(\Om)\hookrightarrow
C(\bar\Om)$. The weak formulation of the state equation \eqref{eq:equation} reads as: For given
control $q\in Q$ find a state $u\in W_0^{1,s'}(\Omega)$ such that
\begin{equation} \label{definition:contproblem2}
  (\nabla u, \nabla \phi) = (q,\phi) \quad \forall \phi \in W_0^{1,s}(\Omega),
\end{equation}
where here and in the sequel, $(\cdot,\cdot)$ denotes the $L^2(\Om)$ inner product.

For the state equation one obtains (cf. \cite{MR3396210}, Theorem 3.2.1.2) the existence and
uniqueness of a solution $u\in H^2(\Om)\cap W^{1,s'}_0(\Omega)\hookrightarrow C(\bar\Om)$.  This allows us to
define a linear control-to-state mapping $S\colon Q \rightarrow C(\bar\Omega)$ as $Sq=u$ where $u$
is the solution of \eqref{definition:contproblem2}. We have the following standard estimate for
$Sq$:
\begin{equation}\label{ellipticregularity}
  \norm{Sq}_{L^\infty(\Om)}\le C \norm{Sq}_{H^2(\Omega)}\leq C \norm{q}_{L^2(\Om)}.
\end{equation}

For a mutually disjoint set of points $\set{x_i |i \in I}\subset \Omega$ with
$I=\set{1,2,\dots,N}\subset \N$ and prescribed target values $\set{\xi_i}_{i\in I}\subset \R$ at
these points, we define the cost functional $J\colon Q\times C(\bar \Omega) \rightarrow \R$ as
\[
  J(q,u)= \frac{1}{2}\sum_{i\in I} (u(x_i)-\xi_i)^2 + \frac{\alpha}{2}\norm{q}^2_{L^2(\Omega)}
\]
We then aim at solving the following optimal control problem:
\begin{equation} \label{definition:contproblem_red}
  \text{Minimize }J(q,u)\text{ subject to~\eqref{definition:contproblem2} and }(q,u)\in \Qad\times C(\bar\Om).
\end{equation}

\begin{theorem} \label{theorem:uniqueness_existence}
  Problem \eqref{definition:contproblem_red} has a unique solution $\bar q\in \Qad$.
\end{theorem}
\begin{proof}
  The proof can be done by standard arguments, cf., e.g.,~\cite{MR2583281}.
\end{proof}

With the reduced cost functional $j\colon Q\rightarrow \R$ given by means of the control-to-state
mapping $S$ as
\[
  j(q)=J(q,Sq),
\]
it is straightforward to see that problem \eqref{definition:contproblem_red} is equivalent to the problem
\begin{equation} \label{reducedproblem}
  \text{Minimize } j(q) \text{ subject to }q \in \Qad.
\end{equation}

\begin{lemma} \label{contderivatives}
  For $q,\delta q\in Q$, the first Fréchet derivative of the reduced function $j$ is given by
  \[
    j'(q)(\delta q)= (\alpha q+z , \delta q),
  \]
  where $z\in W^{1,s}_0(\Omega)$ solves
  \begin{equation}\label{eq:z}
    (\nabla z, \nabla \phi) = \sum_{i \in I} (S q(x_i) -\xi_i)\phi(x_i) \quad \forall \phi \in W_{0}^{1,s'}(\Omega).
  \end{equation}
  The second Fréchet derivative is given for $q,\delta q,\tau q \in Q$ by
  \[
    j''(q)(\delta q,\tau q) = \sum_{i \in I} S\delta q(x_i)S\tau q(x_i) + \alpha (\delta q, \tau q).
  \]
\end{lemma}
\begin{proof}
  The proof is standard. The regularity of $z$ can be found, e.g., in~\cite[Theorem 4]{MR861100}.
\end{proof}
For $i\in I$, let $z_i\in W^{1,s}_0(\Om)$ be given as the solution of
\begin{equation}\label{eq:zi}
  (\nabla z_i,\nabla \phi)=\phi(x_i)\quad\forall\phi\in W^{1,s'}_0(\Om).
\end{equation}
Then, it holds by construction that the solution $z$ of~\eqref{eq:z} can be expressed as
\begin{equation}\label{eq:rep_z}
  z=\sum_{i\in I}(Sq(x_i)-\xi_i)z_i.
\end{equation}

\begin{theorem} \label{theorem:continuousproblem:optimalityconditions}
  A control $\bar q\in \Qad$ with associated state $\bar u=S\bar q \in W_{0}^{1,s'}(\Omega)$ is an
  optimal solution to the problem \eqref{definition:contproblem_red} if and only if there exists an
  adjoint state $\bar z \in W_0^{1,s}(\Omega)$ such that
  \begin{alignat}{2}
    (\nabla\bar u,\nabla \phi)&=(\bar q, \phi) && \forall \phi \in W_0^{1,s}(\Omega),\notag\\
    (\nabla \bar z, \nabla \phi) &= \sum_{i \in I} (\bar u(x_i) -\xi_i)\phi(x_i) &\quad& \forall \phi \in W_{0}^{1,s'}(\Omega),\label{theorem:continuousproblem:adjointeq}\\
    (\alpha \bar q+\bar z, \delta q - \bar q) &\geq 0 &\quad& \forall
    \delta q \in \Qad.\label{theorem:continuousproblem:optcondeq}
  \end{alignat}
\end{theorem}
\begin{proof}
  It holds by a standard result (cf. \cite[Theorem 1.4]{MR0271512}) that $\bar q$ is a solution to
  \eqref{definition:contproblem1} if and only if $j'(\bar q)(\delta q-\bar q)\geq 0$ for all
  $\delta q \in \Qad$. Then, Lemma~\ref{contderivatives} implies the assertion.
\end{proof}

\begin{prop} \label{projectionformula}
  A control $\bar q\in\Qad$ is the solution to the optimal control
  problem~\eqref{definition:contproblem_red} if and only if $\bar q$ and the solution $\bar z$
  of~\eqref{theorem:continuousproblem:adjointeq} fulfill the projection formula
  \[
    \bar q = P_{[a,b]} \left( -\frac{1}{\alpha} \bar z \right).
  \]
  Here the projection $P_{[a,b]}$ is given by
  \[
    P_{[a,b]}(g(x))= \min \left( b, \max \left( a, g(x) \right) \right)
  \]
  for  $g(x) \in \R \cup \{-\infty,\infty\} $ at $x\in \Omega$.
\end{prop}
\begin{proof}
  A proof of the equivalence of~\eqref{theorem:continuousproblem:optcondeq} and this projection
  formula can be found in~\cite[Theorem 2.28]{MR2583281}. 
\end{proof}

\begin{prop}\label{prop:H1}
  For the solution $\bar q\in \Qad$ of the optimal control
  problem~\eqref{definition:contproblem_red}, it holds
  \[
    \bar q\in H^1(\Om).
  \]
\end{prop}
\begin{proof}
  This result follows directly from~\cite[Lemma~3.3]{MR3264224}.
\end{proof}

%%%%%%%%%%%%%%%%%%%%%%%%%%%%%%%%%%%%%%%%%%%%%%%%%%%%%%%%%%%%%%
\section{Discrete Problem}\label{sec:disc}
%%%%%%%%%%%%%%%%%%%%%%%%%%%%%%%%%%%%%%%%%%%%%%%%%%%%%%%%%%%%%%

We approximate the continuous state equation~\eqref{definition:contproblem2} using a Galerkin finite
element discretization. For this discretization, we use a family of triangulations $\set{\Th}$.  A
cell $K\in \Th$ has the diameter $h_K$. The discretization parameter $h$ is given as $h= \max_{K\in
\Th}h_K$.  We also require the triangulation to be regular and quasi-uniform.

For discretizing the state and adjoint equations, we consider the conforming space $V_h\subset
W^{1,\infty}(\Om)$ of linear finite elements on the triangulation $\Th$
\[
  V_h = \Set{ v_h \in C(\bar \Omega) | v_h\vert_K \in \PP_1(K)\; \forall K \in \Th \text{
  and } v_h\bigr\rvert_{\partial\Omega} = 0 }.
\]

We consider two types of discretizations for the control variable. The first type is the so-called
variational discretization introduced by~\cite{MR2122182}. Here, the control variable is not
explicitly discretized.  As second possibility, we consider a piecewise constant control
discretization on the family of triangulations $\set{\Th}$ introduced for the discretization of the
states. Then, we define the space of discrete controls as
\[
  Q_h^c = \Set {q_h \in Q | q_h\vert_K \in \PP_0(K) \text{ for all } K\in \Th}.
\]
where $\PP_0(K)$ denotes the space of piecewise constant polynomials on a cell $K$. The discrete
admissible set is then defined as $\Qhad^c=Q_h^c\cap\Qad$.

In the following, we introduce properties for the discrete problem similar to the continuous case in
Section~\ref{sec:continuousproblem} before. To this end, $\Qhad$ serves as placeholder for either $\Qad$
(for variational discretization) or $\Qhad^c$ (for cellwise constant discretization).

The discrete state equation for $u_h \in V_h$ with given $q \in Q$ reads as
\begin{equation} \label{definitiondiscretestateeq}
  (\nabla u_h, \nabla \phi_h) = (q,\phi_h) \quad \forall \phi_h \in V_h
\end{equation}
and the discrete analog to~\eqref{definition:contproblem_red} has the form
\begin{equation} \label{definition:discreteproblem_red}
  \text{Minimize }J(q_h,u_h)\text{ subject to~\eqref{definitiondiscretestateeq} and }(q_h,u_h)\in
  \Qhad\times V_h.
\end{equation}
Again, we can define the discrete control-to-state mapping with $S_h\colon Q\rightarrow V_h$ as $S_hq=u_h$.
We define the discrete reduced cost functional $j_h\colon Q\rightarrow \R$ by
\[
  j_h(q)= J(q,S_hq).
\]

We start with the following stability result for the discrete solution operator $S_h$.

\begin{lemma} \label{lemma:stability}
  For $q\in Q$, it holds
  \[
    \norm{S_hq}_{L^\infty(\Om)}\le C \norm{q}_{L^2(\Om)}.
  \]
\end{lemma}

\begin{proof}
  The proof is standard. The sub-optimal $L^{\infty}$ error estimate \cite[p.\ 168]{MR1930132} implies
  with~\eqref{ellipticregularity} that
  \[
    \norm{Sq - S_hq}_{L^\infty(\Om)}\leq C h^{2-\frac{d}{2}}\norm{Sq}_{H^2(\Om)} \leq C
    h^{2-\frac{d}{2}} \norm{q}_{L^2(\Om)}.
  \]
  Hence,
  \[
    \norm{S_hq}_{L^\infty(\Om)} \leq \norm{Sq}_{L^\infty(\Om)} + \norm{Sq-S_hq}_{L^\infty(\Om)}\le
    C(1+h^{2-\frac d2})\norm{q}_{L^2(\Om)}
  \]
  implies the assertion.
\end{proof}

\begin{theorem} \label{theorem:uniqueness_existence_discrete}
  Problem \eqref{definition:discreteproblem_red} admits a unique solution $\bar q_h\in \Qhad$.
\end{theorem}

\begin{proof}
  As Theorem~\ref{theorem:uniqueness_existence}, this can be shown by standard arguments.
\end{proof}

As in the continuous case, we have the following expressions for the first and second derivatives of
$j_h$.

\begin{lemma}\label{discderivatives}
  For $q,\delta q \in Q$, the first Fréchet derivative of the reduced function $j$ is given by
  \[
    j_h'(q)(\delta q)= (\alpha q+z_h , \delta q),
  \]
  where $z_h\in V_h$ solves
  \begin{equation}\label{eq:zh}
    (\nabla z_h, \nabla \phi_h) = \sum_{i \in I} (S_h q(x_i) -\xi_i)\phi_h(x_i) \quad \forall \phi_h \in V_h.
  \end{equation}
  The second Fréchet derivative is given for $q,\delta q,\tau q \in Q$ by
  \[
    j_h''(q)(\delta q,\tau q) = \sum_{i \in I} S_h\delta q(x_i)S_h\tau q(x_i) + \alpha (\delta q, \tau q).
  \]
\end{lemma}
For $i\in I$, let $z_{h,i}\in V_h$ be given as the solution of
\begin{equation}\label{eq:zih}
  (\nabla z_{h,i},\nabla \phi_h)=\phi_h(x_i)\quad\forall \phi_h\in V_h.
\end{equation}
Then, it holds by construction that the solution $z_h$ of~\eqref{eq:zh} can be expressed as
\begin{equation}\label{eq:rep_zh}
  z_h=\sum_{i\in I}(S_hq(x_i)-\xi_i)z_{h,i}.
\end{equation}

In the following theorem, we formulate the optimality conditions for the discrete problem. 
\begin{theorem} \label{theorem:discreteroblem:optimalityconditions}
  A control $\bar q_h\in \Qhad$  with associated state $\bar u_h=S_h\bar q_h \in V_h$ is an optimal
  solution to the problem \eqref{definition:discreteproblem_red} if and only if there exists and
  adjoint state $\bar z_h \in
  V_h$ such that
  \begin{alignat}{2}
    (\nabla\bar u_h,\nabla \phi_h)&=(\bar q_h, \phi_h) && \forall \phi_h \in V_h, \notag\\
    (\nabla \bar z_h, \nabla \phi_h) &= \sum_{i\in I} (\bar u_h(x_i) -\xi_i)\phi_h(x_i) &\quad & \forall \phi_h \in V_h,\label{theorem:discreteproblem:adjointeq} \\
    (\alpha \bar q_h+\bar z_h, \delta q_h - \bar q_h) &\geq 0
    \qquad&& \forall \delta q_h \in \Qhad.\label{theorem:discreteproblem:optcondeq}
  \end{alignat}
\end{theorem}

\begin{proof}
  The assertion of the theorem can be proved in the same way as the continuous analogue in
  Theorem~\ref{theorem:continuousproblem:optimalityconditions}.
\end{proof}

\begin{remark}\label{rem:P}
  Similar to Proposition~\ref{projectionformula} on the continuous level,
  condition~\eqref{theorem:discreteproblem:optcondeq} can be formulated by means of the projection
  $P_{[a,b]}$. For $\Qhad=\Qad$, it holds
  \[
    \bar q_h = P_{[a,b]} \left( -\frac{1}{\alpha} \bar z_h \right).
  \]
  In the case $\Qhad=\Qhad^c$, the corresponding formula reads as
  \[
    \bar q_h = P_{[a,b]} \left( -\frac{1}{\alpha} \pi_h\bar z_h \right),
  \]
  where $\pi_h$ denotes the $L^2$ projection on $Q_h^c$.
\end{remark}

%%%%%%%%%%%%%%%%%%%%%%%%%%%%%%%%%%%%%%%%%%%%%%%%%%%%%%%%%%%%%%
\section{Finite Element Error Analysis for an Auxiliary Equation}\label{sec:aux}
%%%%%%%%%%%%%%%%%%%%%%%%%%%%%%%%%%%%%%%%%%%%%%%%%%%%%%%%%%%%%%

For the numerical analysis carried out later, we first need to bound the error $u(x_i)-u_h(x_i)$ for
given $q\in\Qad$. The are multiple results available for the state equation in case of a bounded
right-hand side, see e.g., \cite{MR0474884,MR0488859,MR645661} but mostly for $C^2$ smooth
boundaries in contrast to our case.

For given $f\in L^2(\Om)$ let $w\in H^1_0(\Om)$ be the solution of the auxiliary problem
\begin{equation}\label{eq:w}
  (\nabla w, \nabla \phi )=(f,\phi) \quad \forall \phi\in H_{0}^1(\Omega).
\end{equation}

\begin{lemma}\label{lemma:Lp-W1infty}
  Let $w\in H^1_0(\Om)$ be the solution of~\eqref{eq:w}. Provided that $f\in L^p(\Om)$ with $p>d$,
  it holds $w\in W^{1,\infty}(\Om)$ with
  \[
    \norm{\nabla w}_{L^\infty(\Om)}\le C \norm{f}_{L^p(\Om)}.
  \]
 \end{lemma}

\begin{proof}
  For $d=2$, the assertion follows from~\cite[Theorems~4.3.2.4 and~4.4.3.7]{MR3396210} using the convexity of $\Om$ which ensures the
  existence of $p>2$ with $w\in W^{2,p}(\Om)\hookrightarrow W^{1,\infty}(\Om)$.

  For $d=3$ and given $y\in\Om$, let $G(\cdot,y)$ be the Green's function associated to $y$.  Then,
  there holds by the convexity of $\Om$ and~\cite[Theorem 5.1.8]{MR2641539} that
  \[
    \abs{\nabla_{y} G(x,y)}\leq \abs{x-y}^{-2}\quad\forall x\in\Om,~x\neq y.
  \]
  This allows us to estimate by the H\"older inequality for
  $1=\frac{1}{p}+\frac{1}{p'}$ that
  \[
    \abs{\nabla w(y)}= \left\lvert\int_{\Omega}\nabla_{y} G(x,y)f(x)\,dx\right\rvert
    \leq \int_{\Omega} \abs{f(x)}\abs{x-y}^{-2}\, dx
    \leq \norm{f}_{L^p(\Omega)}\left\lVert\abs{x-y}^{-2}\right\rVert_{L^{p'}(\Omega)}.
  \]
  Since $p>3$, the second factor is bounded, which completes the proof.
\end{proof}

We continue with a regularity result for a the solution $w$ of~\eqref{eq:w} in a sub domain $\Omega_1
\Subset \Omega_0 \Subset \Omega$ provided that the right-hand side $f$ has higher regularity in
$\Omega_0$.

\begin{lemma} \label{lemma:auxiliary}
  Let  $\Omega_1 \Subset \Omega_0 \Subset \Omega$ with  $\Om_0$ smooth and $w$ be the solution
  of~\eqref{eq:w}. 
  \begin{enumerate}[(i)]
    \item Provided that $f\bigr\rvert_{\Omega_0}\in L^p(\Omega_0)$ for some $2\le p<\infty$, 
      it holds $w\in W^{2,p}(\Omega_1)$ and
      \[
        \norm{w}_{W^{2,p}(\Omega_1)}\leq C_p\bigl\{\norm{f}_{L^p(\Omega_0)} + \norm{f}_{L^2(\Om)}\bigr\}
      \]
      where $C_p \sim Cp$ for $p\to\infty$.
    \item Provided that $f\bigr\rvert_{\Omega_0}=0$, it holds $w\in W^{2,\infty}(\Omega_1)$ and
      \[
        \norm{w}_{W^{2,\infty}(\Omega_1)}\leq C \norm{f}_{L^2(\Om)}.
      \]
  \end{enumerate}
\end{lemma}
\begin{proof}
  The proof of the first assertion follows along the lines of the proof of \cite[Lemma
  2.5]{MR3549870}. We choose a smooth $\tilde\Om_1$ such that $\Omega_1 \Subset \tilde\Omega_1
  \Subset\Omega_0$. We first prove that $w\in W^{2,6}(\tilde \Om_1)$. To this end, let $\tilde\omega
  \in [0,1]$ be a smooth cutoff function with the properties
  \[
    \begin{aligned}
      \tilde\omega &=1 &\quad&\text{in } \tilde\Omega_1,\\
      \tilde\omega  &=0 && \text{in }  \Omega\setminus \Omega_0.
    \end{aligned}
  \]
  Set $v$ as $v=\tilde\omega w$. Then, there holds by the product rule
  \[
    \begin{split}
      (\nabla v,\nabla \phi)&=(\nabla(\tilde \omega w),\nabla\phi)=(\tilde \omega\nabla
      w,\nabla\phi)+(w\nabla\tilde \omega,\nabla\phi)\\
      &=(\nabla w,\nabla(\tilde\omega\phi))-(\nabla
      w,\phi\nabla\tilde\omega)-(\nabla\cdot(w\nabla\tilde\omega),\phi)\\
      &=(\tilde\omega f,\phi)-2(\nabla\tilde\omega\nabla w,\phi)-(w\Delta\tilde\omega,\phi)
    \end{split}
  \]
  and therefore $v$ satisfies the following equation
  \[
    \begin{aligned}
      -\Delta v &= \tilde\omega f- 2 \nabla \tilde\omega\nabla w -w\Delta \tilde\omega    &\quad&
      \text{in }\Omega_0,\\
      v &=0 &&\text{on } \partial\Omega_0.
    \end{aligned}
  \]
  By~\cite[Corollary 9.10]{MR1814364}, the $W^{2,6}(\Om_0)$ norm of $v$ is bounded by the $L^6(\Om_0)$
  norm of the  right-hand side above. Using the smoothness of $\tilde\omega$, it follows
  \[
    \begin{split}
      \norm{w}_{W^{2,6}(\tilde\Omega_1)}&=\norm{v}_{W^{2,6}(\tilde\Omega_1)}\leq
      \norm{v}_{W^{2,6}(\Omega_0)}\leq C\norm{\tilde\omega f- 2 \nabla \tilde\omega\nabla w-w\Delta \tilde\omega }_{L^6(\Omega_0)}\\
      &\leq C\bigl\{\norm{f}_{L^6(\Omega_0)}+\norm{\nabla w}_{L^6(\Omega_0)}+\norm{w}_{L^6(\Omega_0)}\bigr\}.
    \end{split}
  \]
  By \eqref{ellipticregularity} and the continuous embedding $w\in H^2(\Omega)\hookrightarrow
  W^{1,6}(\Omega)$ we get
  \[
    \norm{w}_{W^{2,6}(\tilde\Omega_1)}\le C\bigl\{\norm{f}_{L^6(\Omega_0)}+
    \norm{f}_{L^2(\Om)}\bigr\}.
  \]
  For $p\le 6$, this already states the assertion.

  For $p>6$, we iterate the previous steps with
  \[
    \begin{aligned}
      &\omega (x) =1 &\quad& \text{in } \Omega_1,\\
      &\omega (x) =0 && \text{in } \Omega\setminus \tilde\Omega_1,
    \end{aligned}
  \]
  and use the smoothness of $\tilde\Om_1$ to estimate
  \[
    \begin{split}
      \norm{w}_{W^{2,p}(\Omega_1)}=\norm{v}_{W^{2,p}(\Omega_1)}\leq \norm{v}_{W^{2,p}(\tilde\Omega_1)}&\leq
      C_p\norm{\omega  f - 2 \nabla \omega \nabla w-w\Delta \omega  }_{L^p(\tilde\Omega_1)}\\
      &\leq C_p\bigl\{\norm{f}_{L^p(\tilde\Omega_1)}+\norm{\nabla w}_{L^p(\tilde\Omega_1)}+\norm{w}_{L^p(\tilde\Omega_1)}\bigr\},
    \end{split}
  \]
  where $C_p$ can be traced from the proof of \cite[Theorem 9.8, Theorem 9.9]{MR1814364}.  Exploiting
  $W^{2,6}(\tilde\Omega_1)\hookrightarrow W^{1,p}(\tilde\Omega_1)$ for all $1<p<\infty$, we have
  \[
    \norm{w}_{L^p(\tilde\Omega_1)}+\norm{\nabla w}_{L^p(\tilde\Omega_1)}\leq
    \norm{w}_{W^{2,6}(\tilde\Omega_1)}\le C\bigl\{\norm{f}_{L^6(\Omega_0)}+
    \norm{f}_{L^2(\Om)}\bigr\},
  \]
  which concludes the proof of the first assertion.

  The second assertion follows similarly. Noting that $\tilde\omega f=0$ on the whole domain
  $\Omega$ and the smoothness of $\Om_0$, the first step implies $v\in H^3(\Om_0)$ and hence $w\in
  H^3(\tilde\Om_1)$. Then, since also $\omega f=0$ on the whole $\Om$ and $\tilde\Om_1$ is smooth,
  the next step yields $v\in H^5(\tilde\Om_1)$ and consequently $w\in H^5(\Om_1)$. This implies the
  second assertion.
\end{proof}

Let $w_h\in V_h$ being the Ritz projection of $w$ given by
\begin{equation}\label{eq:wh}
  (\nabla w_h, \nabla \phi_h )=(f,\phi_h) \quad \forall v\in V_h.
\end{equation}
We will use the following Schatz-Wahlbin-type estimate to bound the error between the solutions
$w$ of~\eqref{eq:w} and $w_h$ of~\eqref{eq:wh} on a subset of $\Omega$.

\begin{prop} \label{theorem:compactconv}
  Let $\Omega_1 \Subset \Omega_0 \Subset \Omega$, $v\in H^1_0(\Om)\cap C(\bar\Om)$, and $v_h\in V_h$
  satisfying
  \[
    (\nabla (v-v_h),\nabla \phi_h) = 0 \quad \forall \phi_h \in V_h.
  \]
  Then, there are constants $C, C'>0$, $0<h_0<1$, and $r>0$ such that $C'h\leq r$,
  $\operatorname{dist}(\Omega_1,\partial \Omega_0)\geq r$, and
  $\operatorname{dist}(\Omega_0,\partial \Omega)\geq r$. For $0 <h \leq h_0$ and any $\phi_h\in
  V_h$, there holds
  \[
    \norm{v-v_h}_{L^{\infty}(\Omega_1)} \leq C \bigl\{ \abs{\ln
    rh}\norm{v-\phi_h}_{L^{\infty}(\Omega_0)}+r^{-\frac{d}{2}}\norm{v-v_h}_{L^2(\Omega_0)}\bigr\}.
  \]
\end{prop}
\begin{proof}
  See~\cite[Corollary 5.1]{MR0431753}. The corollary there is stated in a slightly different form.
  Applying it to $v-\phi_h-v_h+\phi_h$ yields the stated assertion.
\end{proof}

\begin{lemma} \label{lemma:pointwisebounds}
  Let $\Omega_1\Subset\Omega_0\Subset\Omega$ with smooth $\Om_0$. Further, let $w\in H^1_0(\Omega)$
  be the solution of~\eqref{eq:w} and $w_h\in V_h$ be the solution of~\eqref{eq:wh}. 
  \begin{enumerate}[(i)]
  \item Provided that $f\bigr\rvert_{\Omega_0}\in L^\infty(\Omega_0)$ there is $h_0>0$ such that 
      \[
        \norm{w-w_h}_{L^{\infty}(\Omega_1)}\leq C h^2\lh^2\bigl\{
        \norm{f}_{L^{\infty}(\Omega_0)} + \norm{f}_{L^2(\Om)}\bigr\}.
      \]
      holds for $h\le h_0$.
      \item Provided that $f\bigr\rvert_{\Omega_0}=0$, it holds
      \[
        \norm{w-w_h}_{L^{\infty}(\Omega_1)}\leq C h^2\lh\norm{f}_{L^2(\Om)}.
      \]
  \end{enumerate}
\end{lemma}

\begin{proof}
  We start with the first assertion. By Proposition~\ref{theorem:compactconv}, we have for a
  suitable chosen smooth $\tilde\Om_1$ with $\Omega_1 \Subset \tilde\Omega_1 \Subset\Omega_0$ and
  any $\phi_h\in V_h$
  \[
    \norm{w-w_h}_{L^{\infty}(\Omega_1)}\leq C \bigl\{ \abs{\ln
    rh}\norm{w-\phi_h}_{L^{\infty}(\tilde\Omega_1)}+ r^{-\frac d2}\norm{w-w_h}_{L^{2}(\tilde\Omega_1)}\bigr\}.
  \]
  Since $r$ is constant, we can use a standard result to estimate the second term, i.e.,
  \[
    r^{-\frac d2}\norm{w-w_h}_{L^{2}(\tilde\Omega_1)}\le r^{-\frac
    d2}\norm{w-w_h}_{L^{2}(\Omega)}\le Ch^2\norm{f}_{L^2(\Om)}.
  \]
  For the first term, we note that Lemma~\ref{lemma:auxiliary} ensures $w\in W^{2,p}(\tilde\Om_1)$ for all
  $2\le p<\infty$. For $\phi_h=i_hw$ being the standard nodal interpolant  of $w$, it holds
  \[
    \norm{w-i_hw}_{L^{\infty}(\tilde\Omega_1)} 
    \leq C h^{2-\frac dp}\norm{\nabla^2 w}_{L^p(\tilde\Omega_1)}
    \leq Cp h^{2- \frac dp} \bigl\{\norm{f}_{L^{\infty}(\Omega_0)} + \norm{f}_{L^2(\Om)}\bigr\}.
  \]
  Choosing $p= \abs{\ln h}$ for small $h$, we get $ph^{2-\frac dp}\leq C h^2\abs{\ln h}$, which
  implies the stated estimate.

  The second assertion follows similarly by using the second assertion of
  Lemma~\ref{lemma:auxiliary} ensuring $w\in W^{2,\infty}(\tilde \Om_1)$:
  \[
    \norm{w-i_hw}_{L^{\infty}(\tilde\Omega_1)} \le Ch^2\norm{\nabla^2 w}_{L^\infty(\tilde\Om_1)}\le
    Ch^2\norm{f}_{L^2(\Om)}.
  \]
\end{proof}

%%%%%%%%%%%%%%%%%%%%%%%%%%%%%%%%%%%%%%%%%%%%%%%%%%%%%%%%%%%%%%
\section{Estimates for the continuous and discrete Green's functions}\label{sec:green}
%%%%%%%%%%%%%%%%%%%%%%%%%%%%%%%%%%%%%%%%%%%%%%%%%%%%%%%%%%%%%%

In this section, we consider a point $x_0\in\Om$ and the associated Green's function solving
\begin{equation}\label{eq:g}
  \begin{aligned}
    -\Delta g&=\delta_{x_0} &\quad&\text{in }\Om,\\
    g&=0&&\text{on }\partial\Om.
  \end{aligned}
\end{equation}
It is well-known that $g\in W^{1,s}_0(\Om)$ for any $s<\frac d{d-1}$ with
\begin{equation}\label{eq:reg_g}
  \norm{\nabla g}_{L^s(\Om)}\le C,
\end{equation}
see, e.g.,~\cite[Theorem~4]{MR861100}.

\begin{lemma} \label{lemma:cutoffcont}
  Let $g$ be the solution of~\eqref{eq:g}. Then, for every $M>0$ there exists an open ball $B\subset
  \Om$ containing $x_0$ such that
  \[
    g(x)\geq M \quad\forall x\in B.
  \]
\end{lemma}
\begin{proof}
  It is well known, see, e.g.,~\cite[(3.11)]{MR2434067} that the asymptotic behavior of $g$ for
  $x\to x_0$ is of type
\begin{equation}\label{proof:boundednessofprojection2.1}
    g(x) \approx \begin{cases} 
      C_1 \ln \frac1{\abs{x-x_0}}+C_2 & \text{for } d=2,\\
      C_1 \frac1{\abs{x-x_0}}+C_2 & \text{for } d=3
    \end{cases}
\end{equation}
with $C_1>0$. This directly implies the assertion.
\end{proof}

\begin{lemma} \label{lemma:lipschitzatedges}
  For the solution $g$ of~\eqref{eq:g} and any open ball $B \subset \Omega$ containing $x_0$, it holds
  \[
    g\in W^{1,\infty}(\Omega\setminus \bar B)\cap H^2(\Omega\setminus \bar B).
  \]
\end{lemma}
\begin{proof}
  There exists an open ball $B'\Subset B$ with $x_0\in B'$. As before, we consider a smooth cutoff
  function $\omega\colon\Om\to  [0,1]$ with the following properties
  \[
    \begin{aligned}
      \omega  &=1 &\quad &\text{in } \Omega\setminus \bar B,\\
      \omega  &=0 &&\text{in } B'.
    \end{aligned}
  \]
  Let $\phi\in W^{1,s'}_0(\Om)$. As in the proof of Lemma~\ref{lemma:auxiliary}, it holds for
  $v=\omega g$ that
  \[
      (\nabla v,\nabla \phi)
      =\omega(x_0)\phi(x_0)-2(\nabla\tilde\omega\nabla g,\phi)-(g\Delta\tilde\omega,\phi)
      =-2(\nabla\omega\nabla g,\phi)-(g\Delta\omega,\phi),
  \]
  since $\omega(x_0)=0$. Therefore, $v$ satisfies the following equation
  \[
    \begin{aligned}
      -\Delta v &=  -g\Delta \omega- 2 \nabla \omega\nabla g &\quad& \text{in } \Omega,\\
      v &=0 &&\text{on } \partial\Omega.
    \end{aligned}
  \]
  For any $s<\frac d{d-1}$, we have $g\in W^{1,s}(\Om)$ from~\eqref{eq:reg_g} and it follows $v\in
  W^{2,s}(\Om)$. Hence, by the smoothness of $\omega$, we obtain $g\in W^{2,s}(\Om\setminus \bar
  B)\hookrightarrow H^1(\Om\setminus \bar B)$ .

  Iterating the previous steps using $g\in H^1(\Om\setminus \bar B)$, we obtain $v\in H^2(\Om)$ and
  hence $g\in H^2(\Om\setminus \bar B)\hookrightarrow W^{1,6}(\Om\setminus\bar B)$.

  Iterating again using $g\in  W^{1,6}(\Om\setminus\bar B)$ implies $v\in W^{1,\infty}(\Om)$ by
  Lemma~\ref{lemma:Lp-W1infty}. Consequently, we obtain the assertion $g\in W^{1,\infty}(\Om\setminus
  \bar B)\cap H^2(\Om\setminus \bar B)$.
\end{proof}

Let $g_h\in V_h$ be the Ritz projection of $g$ given as solution of
\begin{equation}\label{eq:gh}
  (\nabla g_h,\nabla \phi_h)=\phi_h(x_0)\quad\forall \phi_h\in V_h.
\end{equation}

\begin{lemma} \label{lemma:l1convergence}
  For the solutions $g\in W^{1,s}_0(\Om)$ of~\eqref{eq:g} and $g_h\in V_h$ of~\eqref{eq:gh},
  it holds 
  \[
    \norm{g-g_h}_{L^1(\Om)}\leq Ch^2 \abs{\ln h}^2,
  \]
  where $C$ is independent of $h$.
\end{lemma}
\begin{proof}
  The assertion is a direct consequence of~\cite[Lemma 3.3 (ii)]{MR3072225}, cf.
  also~\cite{MR0471370} for $d=2$. However, the exponent of the log-term in~\cite[Lemma 3.3 (ii)]{MR3072225} is different from $2$ in the three-dimensional case. Therefore, we give a proof here, which is a direct consequence of Lemma~\ref{lemma:pointwisebounds},~(i).

Let $f=\sgn(g-g_h)$ and let $w \in H^1_0(\Omega)$ and $w_h \in V_h$ be the corresponding solutions of~\eqref{eq:w} and~\eqref{eq:wh}. There holds
\[
\norm{g-g_h}_{L^1(\Om)} =(f,g-g_h) = w(x_0) - w_h(x_0) \le Ch^2 \abs{\ln h}^2 \norm{f}_{L^\infty(\Omega)}\le Ch^2 \abs{\ln h}^2,
\]
where we have used $\norm{f}_{L^\infty(\Omega)} \le 1$.
\end{proof}

\begin{lemma}\label{lemma:boundl2}
  For the solution  $g_h\in V_h$ of~\eqref{eq:gh}, it holds 
  \[
    \norm{g_h}_{L^2(\Om)}\le C
  \]
  with $C$ independent of $h$.
\end{lemma}
\begin{proof}
  We write the discretization error as $e=g-g_h$ and  define $w \in H^1_0(\Omega)$ as the solution of 
  \[
    (\nabla w, \nabla \phi) = (e, \phi) \quad \forall \phi\in H^1_0(\Omega)
  \]
  and $w_h \in V_h$ as the solution of
  \[
    (\nabla w_h, \nabla \phi_h) = (e, \phi_h) \quad \forall \phi_h \in V_h.
  \]
  Then, the $L^2$ error on $\Omega$ can be expressed as
  \[
    \norm{e}^2_{L^2(\Omega)} = (e, g) - (e, g_h) =(\nabla w,\nabla g) - (\nabla w_h,\nabla g_h) 
    =w(x_0)-w_h(x_0)
  \]
  and the sub-optimal $L^{\infty}$ error estimate \cite[p. 168]{MR1930132} implies
  \[
    \norm{e}_{L^2(\Om)}^2\le C h^{2-\frac d2}\norm{e}_{L^2(\Om)}.
  \]
  Hence, we get
  \[
    \norm{g_h}_{L^2(\Om)}\le C h^{2-\frac d2} + \norm{g}_{L^2(\Om)}\le C
  \]
  due to~\eqref{eq:reg_g}.
\end{proof}

\begin{lemma} \label{lemma:convergencewithoutsingularity} 
  For the solutions $g\in W^{1,s}_0(\Om)$ of~\eqref{eq:g} and $g_h\in V_h$ of~\eqref{eq:gh} and any
  open ball $B \subset \Omega$ containing $x_0$, there is $h_0>0$ such that 
  \[
    \norm{g-g_h}_{L^2(\Omega\setminus \bar B)}\leq C h^2 \lh
  \]
  for all $h\le h_0$.
\end{lemma}
\begin{proof}
  We proceeded as in the proof of Lemma~\ref{lemma:boundl2} by writing $e=g-g_h$ and introducing the
  indicator function $\chi_B$ of $B$. Further, we define $w \in H^1_0(\Omega)$ as the solution of 
  \begin{equation}\label{proof:convergencewithoutsingularity3}
    (\nabla w, \nabla \phi) = ((1-\chi_B)e, \phi) \quad \forall \phi\in H^1_0(\Omega)
  \end{equation}
  and $w_h \in V_h$ the solution of
  \[
    (\nabla w_h, \nabla \phi_h) = ((1-\chi_B)e, \phi_h) \quad \forall \phi_h \in V_h.
  \]
  Then, the $L^2$ error on $\Omega \setminus \bar B$ can be expressed as
  \[
    \norm{e}^2_{L^2(\Omega \setminus \bar B)} = ((1-\chi_B)e, g) - ((1-\chi_B)e, z_{h,i})
    =(\nabla w,\nabla g) - (\nabla w_h,\nabla g_h) 
    =w(x_0)-w_h(x_0).
  \]
  Since the right-hand side of \eqref{proof:convergencewithoutsingularity3} is zero on $B$,
  Lemma~\ref{lemma:pointwisebounds} implies by choosing a suitable sub domain $B'\Subset B$
  containing $x_0$ that
  \[
    \norm{e}^2_{L^2(\Omega \setminus \bar B)}=w(x_0)-w_h(x_0) \leq \norm{w-w_h}_{L^{\infty}(B')} \leq C
    h^2 \lh \norm{e}_{L^2(\Omega \setminus \bar B)},
  \]
  which concludes the proof.
\end{proof}

\begin{lemma} \label{lemma:boundednessdiscrete}
  For the solution  $g_h\in V_h$ of~\eqref{eq:gh} and any open ball $B \subset \Omega$ containing
  $x_0$ there is $h_0>0$ such that
  \[
    \norm{g_h}_{L^{\infty}(\Omega \setminus \bar B)}\leq C\quad\text{and}
    \quad \norm{\nabla g_h}_{L^2(\Omega \setminus \bar B)}\leq C,
  \]
  for all $h\le h_0$  with a constant $C$ independent of $h$.
\end{lemma}

\begin{proof}
  On $\Om\setminus \bar B$, we get by an inverse inequality (cf. \cite[Theorem 3.2.6]{MR1930132}),
  Lemma~\ref{lemma:convergencewithoutsingularity}, and by inserting the nodal interpolant $i_hg$ of
  $g$ that
  \[
    \begin{split}
      \norm{g- g_h}_{L^{\infty}(\Omega \setminus \bar B)}
      &\leq \norm{g_h- i_hg}_{L^{\infty}(\Omega \setminus \bar B)}
      + \norm{i_hg-g}_{L^{\infty}(\Omega \setminus \bar B)}\\
      &\leq Ch\norm{\nabla g}_{L^{\infty}(\Omega \setminus \bar B)} +
      Ch^{-\frac{d}{2}}\norm{i_hg-g_h}_{L^2(\Omega \setminus \bar B)}\\
      &\leq Ch\norm{\nabla g}_{L^{\infty}(\Omega \setminus \bar B)} 
      + Ch^{-\frac{d}{2}}\bigl\{\norm{i_hg-g}_{L^2(\Omega \setminus \bar B)} 
      +\norm{g- g_h}_{L^2(\Omega \setminus \bar B)}\bigr\}\\
      &\leq C h^{2-\frac{d}{2}} \lh.
    \end{split}
  \]
  Here, we used that $g\in W^{1,\infty}(\Omega \setminus \bar B)\cap H^2(\Omega\setminus \bar B)$
  according to Lemma~\ref{lemma:lipschitzatedges}. Hence, we get
  \[
    \norm{g_h}_{L^{\infty}(\Omega \setminus \bar B)}\le \norm{g}_{L^{\infty}(\Omega
    \setminus \bar B)}+\norm{g- g_h}_{L^{\infty}(\Omega \setminus \bar B)},
  \]
  which implies the first assertion again by means of Lemma~\ref{lemma:lipschitzatedges}. 
  For the  second assertion, we similarly obtain
  \[
    \begin{split}
      \norm{\nabla(g- g_h)}_{L^2(\Omega \setminus \bar B)}
      &\le Ch\norm{\nabla^2 g}_{L^2(\Omega \setminus \bar B)} 
      + Ch^{-1}\bigl\{\norm{i_hg-g}_{L^2(\Omega \setminus \bar B)} 
      +\norm{g- g_h}_{L^2(\Omega \setminus \bar B)}\bigr\}\\
      &\leq C h\lh.
    \end{split}
  \]
  Here, we again used for the interpolation estimates that $g\in H^2(\Omega \setminus
  \bar B)$. Then,
  \[
    \norm{\nabla g_h}_{L^2(\Omega \setminus \bar B)}\le \norm{\nabla g}_{L^2(\Omega
    \setminus \bar B)}+\norm{\nabla(g- g_h)}_{L^2(\Omega \setminus \bar B)}
  \]
  together with Lemma~\ref{lemma:lipschitzatedges} implies the second assertion.
\end{proof}

We close this section by a discrete analogue of Lemma~\ref{lemma:cutoffcont} for $d=2$ only.

\begin{lemma} \label{lemma:cutoffeq}
  Let $d=2$ and $g_h\in V_h$ be the solution of~\eqref{eq:gh}.  Then, for every $M>0$ there exists an open
  ball $B\Subset \Om$ containing $x_0$ and $h_0>0$ such that for all $h\le h_0$, it holds
  \[
    g_h(x)\geq M \quad\forall x\in B.
  \]
\end{lemma}
\begin{proof}
  The main idea for the proof stems from the proof of~\cite[Theorem 4.6]{MR3614014}. While
  in~\cite{MR3614014}, smoothness of the domain is required we do not need this for our
  particular result since $x_0$ lies in the interior of $\Omega$ and we are only interested in the
  behavior of $g_h$ in a neighborhood of $x_0$ and not close to the boundary.  Hence, case 3 of the
  proof of \cite[Theorem 4.6]{MR3614014} does not need to be considered here. We adapt the
  technique to our case.

  We distinguish the following two cases:
  \begin{description}
    \item[\boldmath$\abs{x- x_0}\ge \kappa h \abs{\ln h}^{\frac{1}{2}}$:] We show, that for this
      case we have pointwise convergence of $g_h$. This setting fulfills the assumptions
      of~\cite[Theorem 6.1]{MR0431753} which states the existence of constants $\kappa$ and
      $C_\kappa$ such that for $h$ sufficiently small and
      $x,x_0\in\Omega_0 \Subset \Omega$ with $\abs{x-x_0} \geq  \kappa h$, it holds
      \begin{equation}\label{eq:6.1}
        \abs{g(x)-g_h(x)}\leq C_\kappa\frac{h^2}{\abs{x-x_0}^2}\ln\frac{\abs{x-x_0}}{h}.
      \end{equation}

      Abbreviating $\eta= \abs{x-x_0}h^{-1}$, the right-hand side
      of~\eqref{eq:6.1} becomes $C_\kappa \eta^{-2}\ln \eta$. Since $\kappa h \abs{\ln h}^{\frac{1}{2}}\leq \abs{x-x_0}$, we have that
      $\eta\geq \kappa \abs{\ln h}^{\frac{1}{2}}$ which means that for $h$ small enough $C_\kappa \eta^{-2}\ln
      \eta$ is maximal at $\eta=\kappa  \abs{\ln h}^{\frac{1}{2}}$. Combined, we have
      \[
        \abs{g(x)- g_h(x)} \le C_\kappa \eta^{-2}\ln \eta\leq  C_\kappa  \abs{\ln h}^{-1} \ln\abs{\ln
        h}^{\frac{1}{2}}.
      \]
      Hence, for $h\le h_0$ sufficiently small, we obtain
      \[
        g_h(x)=g(x)+(g_h(x)-g(x))\ge \frac12 g(x),
      \]
      which implies the assertion in this case  by Lemma~\ref{lemma:cutoffcont}.
    \item[\boldmath$\abs{x-x_0} \leq \kappa h \abs{\ln h}^{\frac{1}{2}}$:] From~\cite[Lemma
      4.8]{MR3614014}, we know that $g_h(x_0) \geq C(1+\abs{\ln h})$.  Due to~\cite[Lemma
      4.9]{MR3614014}, the first derivative of $g_h$ is bounded by $\norm{\nabla
      g_h}_{L^\infty(\Om)}\leq Ch^{-1}$. Then, by the mean value theorem, it holds
      \[
        g_h(x)\ge g_h(x_0)-\norm{\nabla g_h}_{L^\infty(\Om)}\abs{x-x_0}\ge
        C\lh-C\lh^{\frac12}.
      \]
      That implies $\abs{g_h(x)}\geq M$ for $h\le h_0$ sufficiently small.

  \end{description}
  Combination of these cases yields the assertion with a set $B$ which can be chosen independently
  of $h\le h_0$.
\end{proof}

\begin{remark}
  To our best knowledge, the question is open, if it is possible to prove a similar estimate for the
  discrete Green's function $g_h$ in the three-dimensional case ($d=3$) on general quasi-uniform
  meshes.
\end{remark}

%%%%%%%%%%%%%%%%%%%%%%%%%%%%%%%%%%%%%%%%%%%%%%%%%%%%%%%%%%%%%%
\section{Error Analysis for the Optimal Control Problem}\label{sec:error}
%%%%%%%%%%%%%%%%%%%%%%%%%%%%%%%%%%%%%%%%%%%%%%%%%%%%%%%%%%%%%%

In this section, we derive estimates for the discretization error between the continuous optimal
control $\bar q\in \Qad$ and the discrete optimal control $\bar q_h\in\Qhad$ for the case of
variational control (i.e. $\Qhad=\Qad$) and piecewise constant control discretization (i.e.
$\Qhad=\Qhad^c$). In both cases, we derive error estimates of order $h\lh$, cf. the
Theorems~\ref{theorem:result1} and~\ref{theorem:result2}.

Before doing so, we introduce an additional piece of notation. So far we have only discussed the
parts of the adjoint equation related to the singular behavior. Now we also consider the coefficient
$S\bar q(x_i)-\xi_i$ accompanying each $z_i$. If this coefficient becomes zero, the
singularity at that point vanishes.  For the rest of the paper we define 
\[
  L=\Set{ i\in I | S\bar q(x_i)-\xi_i=0}
\]
to be the set of such indices.

\begin{lemma} \label{lemma:boundednessofprojection}
  Let $\bar q\in \Qad$ be the solution of~\eqref{definition:contproblem_red}. Then, for each $i\in I
  \setminus L$ there is an open ball $B_i\subset \Om$ containing $x_i$ such that, depending on the
  sign of $S\bar q(x_i)-\xi_i$, either $\bar q(x)= a$  or $\bar q(x) =b$ holds for all $x\in B_i$ .
  Moreover, it holds $\bar q\in W^{1,\infty}(\Om)$ with
  \[
    \norm{\nabla \bar q}_{L^\infty(\Om)}\le C.
  \]
\end{lemma}
\begin{proof}
  For the solution $\bar z$ of the adjoint equation~\eqref{theorem:continuousproblem:adjointeq}, it
  holds
  \[
    \bar z=\sum_{i\in I}(S\bar q(x_i)-\xi_i)z_i
  \]
  with $z_i$ being the solution of~\eqref{eq:zi}.  The Lemmas~\ref{lemma:cutoffcont}
  and~\ref{lemma:lipschitzatedges} applied to $z_i$ for any $i\in I\setminus L$ ensure that for such
  $i$ and any $M>0$ there are open balls $B_i$ containing $x_i$ such that 
  \[
    \abs{z_i(x)}\ge M\quad\forall x\in B_i\qquad\text{and}\qquad
    \norm{z_j}_{L^\infty(B_i)}\le C\quad\text{for }j\in I\setminus\set{i}.
  \]
  Hence, by using $S\bar q(x_i)-\xi_i\neq0$, we can choose $B_i$ for $i\in I\setminus L$ such that
  either
  \[
    -\frac1\alpha \bar z=-\frac1\alpha \sum_{j\in I}(S\bar q(x_j)-\xi_j)z_j\le a
    \quad\text{or}\quad
    -\frac1\alpha \bar z=-\frac1\alpha \sum_{j\in I}(S\bar q(x_j)-\xi_j)z_j\ge b
  \]
  holds. Hence, by Proposition~\ref{projectionformula}, we obtain the first assertion.

  For the second assertion, we note that  $\bar q\in H^1(\Om)$ by Proposition~\ref{prop:H1}. Hence,
  to prove $\bar q\in W^{1,\infty}(\Om)$ it is sufficient to ensure $\bar q\in W^{1,\infty}(B_i)$
  for every $i\in I\setminus L$ and $\bar q\in W^{1,\infty}(\Om\setminus \bigcup_{i\in I\setminus
  L}\bar B_i)$. The first result follows directly from the previous discussion.
  Lemma~\ref{lemma:lipschitzatedges} yields $\bar q\in W^{1,\infty}(\Om\setminus \bigcup_{i\in
  I\setminus L}\bar B_i)$, which completes the proof.
\end{proof}

As preparation for deriving the error estimates for $\norm{\bar q-\bar q_h}_{L^2(\Om)}$, we prove
the following three lemmas.

\begin{lemma} \label{lemma:normestimateby2ndderivative}
  For $p,q,q_h \in Q$, it holds
  \[
    \alpha\norm{q-q_h}^2_{L^2(\Om)} \leq j_h''(p)(q-q_h,q-q_h).
  \]
\end{lemma}
\begin{proof}
  This follows directly from Lemma~\ref{discderivatives}.
\end{proof}

\begin{lemma} \label{lemma:j-jh}
  For $q,\delta q \in \Qad$, it holds
  \[
    \abs{j'(q)(\delta q) - j'_h(q)(\delta q)}\leq Ch^2\abs{\ln h}^2\norm{\delta q}_{L^\infty(\Om)}
  \]
  with a constant $C$ independent of $h$.
\end{lemma}
\begin{proof}
  By the Lemmas~\ref{contderivatives} and \ref{discderivatives}, we
  arrive at
  \begin{equation}\label{eq:X}
    \abs{j'(q)(\delta q) - j_h'(q)(\delta q)} = \abs{(z - z_h, \delta q )}\le \abs{(z - \tilde z_h, \delta q
    )}+\abs{(\tilde z_h - z_h, \delta q )},
  \end{equation}
  where $z\in W^{1,s}_0(\Om)$ and $z_h\in V_h$ are the solutions of~\eqref{eq:z} and~\eqref{eq:zh},
  respectively and $\tilde z_h\in V_h$ denotes the solution of
  \[
    (\nabla \tilde z_h,\nabla \phi_h)=\sum_{i\in I}(Sq(x_i)-\xi_i)\phi_h(x_i)\quad\forall\phi_h\in V_h.
  \]
  By construction, it holds
  \[
    \tilde z_h=\sum_{i\in I}(Sq(x_i)-\xi_i)z_{h,i}
  \]
  with $z_{h,i}$ given by~\eqref{eq:zih}.

  For the first term on the right-hand side of~\eqref{eq:X}, we get by~\eqref{eq:rep_z} and
  Lemma~\ref{lemma:l1convergence}
  \[
    \begin{split}
      \abs{(z - \tilde z_h, \delta q)}
      &\le \norm{z-\tilde z_h}_{L^1(\Om)}\norm{\delta q}_{L^\infty(\Om)}
      \le \sum_{i\in I}\abs{Sq(x_i)-\xi_i}\norm{z_i-z_{h,i}}_{L^1(\Om)}\norm{\delta q}_{L^\infty(\Om)}\\
      &\le Ch^2\lh^2\bigl\{\norm{Sq}_{L^\infty(\Om)}+\abs{\xi}\bigr\}\norm{\delta q}_{L^\infty(\Om)}\le
      Ch^2\lh^2\bigl\{\norm{q}_{L^2(\Om)}+\abs{\xi}\bigr\}\norm{\delta q}_{L^\infty(\Om)}
    \end{split}
  \]
  with $\abs{\xi}^2=\sum_{i\in I}\xi_i^2$.

  For estimating the second term on the right-hand side of~\eqref{eq:X}, let
  $\Omega_1\Subset\Omega_0\Subset\Omega$ with smooth $\Om_0$ such that $\set{x_i|i\in I}\subset\Om_1$. Then,
  \eqref{eq:rep_zh} and the Lemmas~\ref{lemma:boundl2} and~\ref{lemma:pointwisebounds} yield
  \[
    \begin{split}
      \abs{(\tilde z_h - z_h, \delta q )}
      &\le\norm{\tilde z_h-z_h}_{L^2(\Om)}\norm{\delta q}_{L^2(\Om)}
      \le \sum_{i\in I}\abs{Sq(x_i)-S_hq(x_i)}\norm{z_{h,i}}_{L^2(\Om)}\norm{\delta q}_{L^2(\Om)}\\
      &\le C \norm{Sq-S_hq}_{L^\infty(\Om_1)}\norm{\delta q}_{L^2(\Om)}
      \le C h^2\lh^2\norm{q}_{L^\infty(\Om)}\norm{\delta q}_{L^2(\Om)}.
    \end{split}
  \]
  
  Inserting this back in~\eqref{eq:X} proves the lemma.
\end{proof}

\begin{lemma} \label{lemma:jh-jh}
  Let $q,p,\delta q\in Q$. Then, it holds
  \[
    \abs{j_h'(q)(\delta q) - j_h'(p)(\delta q)}\leq C \norm{q-p }_{L^2(\Omega)} \norm{\delta
    q}_{L^2(\Omega)}.
  \]
\end{lemma}

\begin{proof}
  Due to Lemma~\ref{discderivatives}, it holds by the mean value theorem for any $\rho\in Q$ that
  \[
    j_h'(q)(\delta q) - j_h'(p)(\delta q) = j_h''(\rho)(q-p,\delta q)=\sum_{i\in I}
    S_h(q-p)(x_i)S_h\delta q(x_i) + \alpha (q-p,\delta q).
  \]
  Therefore, using Lemma~\ref{lemma:stability}, we can estimate 
  \[
    \begin{split}
      \abs{j_h'(q)(\delta q) - j_h'(p)(\delta q)}
      &\le C \norm{S_h(q-p)}_{L^\infty(\Om)}\norm{S_h\delta q}_{L^\infty(\Om)} + \alpha
      \abs{(q-p,\delta q)}\\
      &\leq C \norm{q-p}_{L^2(\Om)}\norm{\delta q}_{L^2(\Om)},
    \end{split}
  \]
  which states the estimate.
\end{proof}

%%%%%%%%%%%%%%%%%%%%%%%%%%%%%%%%%%%%%%%%%%%%%%%%%%%%%%%%%%%%%%
\subsection{Variational Discretization}
%%%%%%%%%%%%%%%%%%%%%%%%%%%%%%%%%%%%%%%%%%%%%%%%%%%%%%%%%%%%%%

As first discretization, we consider the variational discretization approach as introduced
by~\cite{MR2122182}, i.e., we choose $\Qhad=\Qad$. Note, that the following error estimate
extends~\cite[Theorem 5.2]{MR3523574} to domains with polygonal or polyhedral boundary and for $d=3$
improves convergence rate by $h^{\frac{1}{2}}$ compared to~\cite[Theorem 5.2]{MR3523574}
and~\cite[Theorem 3.2]{MR3449612}.

\begin{theorem} \label{theorem:result1}
  Let $\bar q\in \Qad$ be the solution of the continuous problem \eqref{definition:contproblem_red}
  and $\bar q_h\in \Qad$ the solution of the corresponding discrete problem
  \eqref{definition:discreteproblem_red} with $\Qhad=\Qad$. Then, it holds
  \[
    \norm{\bar q - \bar q_h}_{L^2(\Om)}\leq C h \lh
  \]
  with a constant $C$ independent of $h$.
\end{theorem}

\begin{proof}
  Using Lemma~\ref{lemma:normestimateby2ndderivative}, we have by the optimality
  conditions~\eqref{theorem:continuousproblem:optcondeq}
  and~\eqref{theorem:discreteproblem:optcondeq} that
  \[
    \begin{split}
      \alpha \norm{\bar q- \bar q_h}^2_{L^2(\Omega)}
      &\leq j_h''(p)(\bar q-\bar q_h,\bar q-\bar q_h) = j_h'(\bar q)(\bar q - \bar q_h)- j'_h(\bar q_h) (\bar q -\bar q_h )\\
      &\leq j_h'(\bar q)( \bar q - \bar q_h) - j'(\bar q)(\bar q - \bar q_h)
    \end{split}
  \]
  for arbitrary $p\in Q$. Since $\bar q,\bar q_h\in\Qad$ are bounded, Lemma~\ref{lemma:j-jh} implies
  \[
    \alpha \norm{ \bar q-\bar q_h }^2_{L^2(\Omega)}
    \leq  j_h'(\bar q)( \bar q - \bar q_h) - j'(\bar q)(\bar q - \bar q_h) 
    \leq Ch^2\abs{\ln h}^2 \norm{\bar q - \bar q_h}_{L^\infty(\Om)}
    \le Ch^2\abs{\ln h}^2,
  \]
  which yields the result.
\end{proof}

%%%%%%%%%%%%%%%%%%%%%%%%%%%%%%%%%%%%%%%%%%%%%%%%%%%%%%%%%%%%%%
\subsection{Cellwise Constant Control Discretization}
%%%%%%%%%%%%%%%%%%%%%%%%%%%%%%%%%%%%%%%%%%%%%%%%%%%%%%%%%%%%%%

We now consider the fully discretized case where the discrete state and adjoint are approximated by
functions in $V_h$ and the discrete control is searched for in $\Qhad=\Qhad^c$.  Note, that the
following error estimate improves for $d=3$ the---to our knowledge---best known error estimate
from~\cite[Theorem 4.3]{AnOtSa} by $h^{\frac{1}{2}}$.

\begin{theorem} \label{theorem:result2}
  Let $\bar q\in \Qad$ be the solution of the continuous problem \eqref{definition:contproblem_red}
  and $\bar q_h\in \Qhad$ the corresponding solution of the discrete problem
  \eqref{definition:discreteproblem_red} with $\Qhad=\Qhad^c$. Then, it holds
  \[
    \norm{\bar q - \bar q_h}_{L^2(\Om)}\leq C h \abs{\ln h}
  \]
  with a constant $C$ independent of $h$.
\end{theorem}

\begin{proof}
  We split the error as
  \[
    \norm{\bar q - \bar q_h}_{L^2(\Om)}\leq \norm{\bar q - \pi_h\bar q}_{L^2(\Om)}+ \norm{\pi_h\bar
    q - \bar q_h}_{L^2(\Om)},
  \]
  where $\pi_h$ denotes the $L^2$ projection on $Q_h^c$ given for $v\in L^1(\Om)$ by 
  \begin{equation}\label{eq:pih}
    (\pi_h v)\bigr\rvert_{K} = \frac{1}{\abs{K}}\int_{K}v\, dx \qquad \forall K\in \Th.
  \end{equation}
  Standard estimates yield by Lemma~\ref{lemma:boundednessofprojection} that 
  \begin{equation}\label{eq:piL2}
    \norm{\bar q - \pi_h \bar q}_{L^2(\Om)} \leq Ch\norm{\nabla \bar q}_{L^2(\Om)} \le Ch.
  \end{equation}

  To bound $\norm{\pi_h\bar q - \bar q_h}_{L^2(\Om)}$, note that $\pi_h\bar q\in \Qhad^c$. We derive
  by means of Lemma~\ref{lemma:normestimateby2ndderivative}, as well as the optimality
  conditions~\eqref{theorem:continuousproblem:optcondeq}
  and~\eqref{theorem:discreteproblem:optcondeq} that
  \[
    \begin{split}
      \alpha\norm{\pi_h \bar q - \bar q_h}^2_{L^2(\Om)}
      &\leq j''_h(p)(\pi_h \bar q -\bar  q_h,\pi_h \bar q - \bar q_h) = j'_h(\pi_h\bar q)(\pi_h \bar q - \bar q_h) - j'_h(\bar q_h)(\pi_h \bar q -\bar q_h)\\
      &\leq j'_h(\pi_h\bar q)(\pi_h \bar q - \bar q_h) - j'(\bar q)(\bar q - \bar q_h).
    \end{split}
  \]
  By further splitting, we obtain
  \begin{equation}\label{proof:result2eq4}
    \begin{split}
      \alpha\norm{\pi_h \bar q - \bar q_h}^2_{L^2(\Om)}
      &\leq j'_h(\pi_h\bar q)(\pi_h \bar q - \bar q_h) - j'(\bar q)(\pi_h \bar q -  \bar q_h) - j'(\bar q)(\bar q - \pi_h \bar q) \\
      &= \bigl\{j'_h(\pi_h\bar q)(\pi_h \bar q - \bar q_h) - j'_h(\bar q)(\pi_h \bar q - \bar q_h\bigr\} \\
      &\quad + \bigl\{j'_h(\bar q)(\pi_h \bar q - \bar q_h) -  j'(\bar q)(\pi_h \bar q - \bar q_h)\bigr\} - j'(\bar q)(\bar q -  \pi_h \bar q).
    \end{split}
  \end{equation}
  For the first term in \eqref{proof:result2eq4} we apply Lemma~\ref{lemma:jh-jh}
  and~\eqref{eq:piL2} to end up with
  \[
    \begin{split}
      \abs{j'_h(\pi_h\bar q)(\pi_h \bar q - \bar q_h) - j'_h(\bar q)(\pi_h \bar q - \bar q_h)} 
      &\leq C \norm{\pi_h \bar q -\bar q}_{L^2(\Om)}\norm{\pi_h \bar q - \bar q_h}_{L^2(\Om)}\\
      &\le Ch^2+\frac\alpha2 \norm{\pi_h \bar q - \bar q_h}_{L^2(\Om)}^2.
    \end{split}
  \]
  The second difference in~\eqref{proof:result2eq4} can be dealt with as in
  Theorem~\ref{theorem:result1}. We apply Lemma~\ref{lemma:j-jh} leading by $\norm{\pi_h\bar
  q}_{L^\infty(\Om)}\le \norm{\bar q}_{L^\infty(\Om)}$ to
  \[
    \begin{split}
      \abs{j'_h(\bar q)(\pi_h \bar q - \bar q_h) -  j'(\bar q)(\pi_h \bar q - \bar q_h)}
      &\leq C h^2 \abs{\ln h}^2\norm{\pi_h\bar q - \bar q_h}_{L^\infty(\Om)}\\
      &\leq C h^2 \abs{\ln h}^2\bigl\{\norm{\bar q}_{L^\infty(\Om)}+\norm{\bar q_h}_{L^\infty(\Om)}\bigr\}
      \le Ch^2\abs{\ln h}^2.
    \end{split}
  \]
  By the $L^2$ orthogonality of the projection $\pi_h$, the last term in~\eqref{proof:result2eq4} amounts to
  \[
    \begin{split}
      \abs{j'(\bar q)(\bar q - \pi_h \bar q)} 
      &= \abs{(\alpha \bar q + \bar z, \bar q - \pi_h \bar q)} = \abs{((\alpha \bar q + \bar z) -\pi_h (\alpha \bar q + \bar z), \bar q - \pi_h \bar q)} \\
      &\leq \norm{(\alpha \bar q + \bar z) -\pi_h (\alpha \bar q + \bar z)}_{L^1(\Om)} \norm{\bar q - \pi_h \bar q}_{L^\infty(\Om)}.
    \end{split}
  \]
  Since $\bar z$ is in $W^{1,1}(\Omega)$ (cf.~\cite[Theorem 4]{MR861100}), we can use standard
  estimates for the $L^2$ projection in $L^1(\Om)$ and $L^\infty(\Om)$ to get
  \[
    \abs{j'(\bar q)(\bar q - \pi_h \bar q)}
    \leq Ch^2 \norm{\nabla (\alpha \bar q+\bar z)}_{L^1(\Om)}\norm{\nabla \bar q}_{L^\infty(\Om)} 
    \leq C h^2 \bigl\{\norm{\nabla \bar z}_{L^1(\Om)}^2+\norm{\nabla \bar
    q}^2_{L^{\infty}(\Omega)}\bigr\}\le Ch^2
  \]
  by~\eqref{eq:reg_g} applied to $\bar z$ and Lemma~\ref{lemma:boundednessofprojection}.
  Collecting the estimated for the right-hand side of~\eqref{proof:result2eq4} and absorbing the
  term $\frac\alpha2 \norm{\pi_h \bar q - \bar q_h}_{L^2(\Om)}^2$ to the left-hand side, yields
  \[
    \alpha\norm{\pi_h \bar q - \bar q_h}^2_{L^2(\Om)}\le Ch^2\abs{\ln h}^2,
  \]
  which implies the assertion together with~\eqref{eq:piL2}.
\end{proof}

%%%%%%%%%%%%%%%%%%%%%%%%%%%%%%%%%%%%%%%%%%%%%%%%%%%%%%%%%%%%%%
\section{Improved Error Analysis for the Optimal Control Problem in two Dimensions}\label{sec:imp2}
%%%%%%%%%%%%%%%%%%%%%%%%%%%%%%%%%%%%%%%%%%%%%%%%%%%%%%%%%%%%%%

In this section, we improve the error estimates derived in Section~\ref{sec:error} using a more
detailed analysis of the behavior of the discrete optimal control in two space dimensions. Moreover,
when discretizing the controls by cellwise constants (i.e. for $\Qhad=\Qhad^c$), we employ a post
processing strategy to overcome the limitations to first order convergence in this case.  For
variational discretization and for cellwise constant control discretization with post processing, we
derive error estimates of order $h^2\lh^2$, cf. the Theorems~\ref{theorem:result1.1}
and~\ref{theorem:result5}.

Throughout this section, the analysis is restricted to $d=2$. We start by proving a result for the
discrete optimal control $\bar q_h$ similar to Lemma~\ref{lemma:boundednessofprojection}. This is
possible due to the Lemmas~\ref{lemma:boundednessdiscrete} and~\ref{lemma:cutoffeq}.

\begin{lemma} \label{lemma:boundednessofprojectiondiscrete}
  Let $d=2$. Further, let $\bar q_h\in \Qhad$ be the solution
  of~\eqref{definition:discreteproblem_red}. Then, for each $i\in I \setminus L$ there is an open
  ball $B_i\subset\Om$ containing $x_i$ such that, depending on the sign of $S_h\bar q(x_i)-\xi_i$,
  either $\bar q_h(x)= a$  or $\bar q_h(x) =b$ holds for all $x\in B_i$ and $h\le h_0$.
\end{lemma}

\begin{proof}
  We first show that for $i\in I\setminus L$ the difference $\abs{S_h \bar q_h
  (x_i) - \xi_i}$ is bounded away from zero provided that $h$ is sufficiently small. For $i\in
  I\setminus L$, it holds
  \[
    \begin{split}
      0
      &<\abs{S\bar q (x_i) - \xi_i}\\
      &\leq \abs{S_h\bar q_h(x_i)-S\bar q_h (x_i)}+\abs{S(\bar q_h - \bar q)(x_i)}+\abs{S_h\bar q_h(x_i)-\xi_i}\\
      &\leq Ch^2 \abs{\ln h} \norm{\bar q_h}_{L^\infty(\Om)}+C\norm{\bar q - \bar
    q_h}_{L^2(\Om)}+\abs{S_h\bar q_h(x_i)-\xi_i}\\
     & \le Ch\lh +\abs{S_h\bar q_h(x_i)-\xi_i}.
    \end{split}
  \]
  Here we used Lemma~\ref{lemma:pointwisebounds},~\eqref{ellipticregularity}, and
  Theorem~\ref{theorem:result2}.  Thereby, we
  conclude that there is $h_0>0$ such that $S_h\bar q_h(x_i)-\xi_i\neq0$ for $i\in I\setminus L$ and $h<h_0$.

  For the solution $\bar z_h\in V_h$ of the discrete adjoint
  equation~\eqref{theorem:discreteproblem:adjointeq}, it holds
  \[
    \bar z_h=\sum_{i\in I}(S_h\bar q_h(x_j)-\xi_j)z_{h,i}
  \]
  with the solutions $z_{h,i}\in V_h$ of~\eqref{eq:zih}.  Lemma~\ref{lemma:cutoffeq} in combination
  with Lemma~\ref{lemma:boundednessdiscrete} applied to $z_{h,i}$ for any $i\in I\setminus L$
  ensures (by possibly reducing $h_0$) that for such $i$ and any $M>0$ there are open balls $B_i$
  containing $x_i$ such that
  \[
    \abs{z_{h,i}(x)}\ge M\quad\forall x\in B_i\qquad\text{and}\qquad
    \norm{z_{h,j}}_{L^\infty(B_i)}\le C\quad\text{for }j\in I\setminus\set{i}
  \]
  for all $h<h_0$. Hence, by using $S_h\bar q_h(x_i)-\xi_i\neq0$, we obtain on $B_i$ for $i\in
  I\setminus L$ that either
  \[
    -\frac1\alpha \bar z_h=-\frac1\alpha \sum_{j\in I}(S_h\bar q_h(x_j)-\xi_j)z_{h,j} \le a
    \quad\text{or}\quad
    -\frac1\alpha \bar z_h=-\frac1\alpha \sum_{j\in I}(S_h\bar q_h(x_j)-\xi_j)z_{h,j}\ge b
  \]
  holds. In the case of variational discretization, i.e., for $\Qhad=\Qad$, the assertion follows
  immediately from Remark~\ref{rem:P} as in Lemma~\ref{lemma:boundednessofprojection}. For cellwise
  discrete control discretization where $\Qhad=\Qhad^c$, let $\pi_h$ the $L^2$ projection on $Q_h^c$
  as given by~\eqref{eq:pih}.  Then, again by Remark~\ref{rem:P}, it holds for $K\in\Th$ that
  \[
    \bar q_h\bigr\rvert_K=P_{[a,b]}\left(-\frac1\alpha (\pi_h\bar z_h)\bigr\rvert_K\right)
    =P_{[a,b]}\left(-\frac1\alpha \bar z_h(S_K)\right),
  \]
  where $S_K$ denotes the centroid of the cell $K$, cf. Section~\ref{sec:post_p}. This implies the
  assertion also in this case.
 \end{proof}

\begin{lemma} \label{lemma:equalityqqh}
  Let $d=2$, $\bar q\in \Qad$ be the solution of \eqref{definition:contproblem_red}, and $\bar
  q_h\in \Qhad$ the solution of \eqref{definition:discreteproblem_red}. Then, there exist $h_0>0$
  and open balls $B_i\subset \Omega$ containing $x_i$ for $i\in I\setminus L$, such that for all
  $h\le h_0$, it holds either
  \[
    \bar q(x) = \bar q_h(x)=a \qquad \text{or}\qquad \bar q(x) = \bar q_h(x)=b  \qquad \forall x\in B_i.
  \]
\end{lemma}
\begin{proof}
  The assertion is a direct consequence of the Lemmas~\ref{lemma:boundednessofprojection}
  and~\ref{lemma:boundednessofprojectiondiscrete}.
\end{proof}

\begin{lemma} \label{lemma:neu}
  Let $d=2$ and $B_i\subset\Om$ open balls with $x_i\in B_i$ for $i\in I\setminus L$. For $q \in
  \Qad$, let $z$  be the solution of~\eqref{eq:z} and $\tilde z_h\in V_h$ be given as the
  solution of
  \begin{equation}\label{eq:tilde_zh}
    (\nabla \tilde z_h,\nabla \phi_h)=\sum_{i\in I}(Sq(x_i)-\xi_i)\phi_h(x_i)\quad\forall\phi_h\in V_h.
  \end{equation}
  Then, it holds for $B=\bigcup_{i\in I\setminus L}B_i$ that
  \[
    \norm{z-\tilde z_h}_{L^2(\Om\setminus \bar B)} \leq Ch^2\lh^2,
  \]
  with a constant $C$ independent of $h$.
\end{lemma}
\begin{proof}
  Let $z_i$ be defined by~\eqref{eq:zi} and $z_{h,i}$ be defined by~\eqref{eq:zih}. We begin with
  the splitting
  \begin{multline}\label{eq:sum}
    \norm{ z-\tilde z_h}_{L^2(\Om\setminus \bar B)}
    \le \sum_{i \in L}\norm{(S q (x_i) -\xi_i)z_i - (S_h q (x_i)
    -\xi_i)z_{h,i}}_{L^2(\Om\setminus \bar B)}\\
    +\sum_{i \in I\setminus L}\norm{(S q (x_i) -\xi_i)z_i - (S_h q (x_i)
    -\xi_i)z_{h,i}}_{L^2(\Om\setminus \bar B)}.
  \end{multline}

  To discuss the first term of~\eqref{eq:sum}, we note that for $i\in L$ there holds $S q(x_i)
  -\xi_i=0$.  Therefore, we have for such $i$ that
  \[
    (S q (x_i) -\xi_i)z_i - (S_h q (x_i) -\xi_i)z_{h,i}
    =(S_h q (x_i) -\xi_i)z_{h,i}
    =(S_h q (x_i) -S q(x_i))z_{h,i}.
  \]
  Then, we get by the Lemmas~\ref{lemma:pointwisebounds} and~\ref{lemma:boundl2} that 
  \begin{equation}\label{eq:2}
    \begin{split}
      \norm{(S q (x_i) -\xi_i)z_i - (S_h q (x_i)-\xi_i)z_{h,i}}_{L^2(\Om\setminus \bar B)}
      &=\norm{(S_h q (x_i) -S q(x_i))z_{h,i}}_{L^2(\Om\setminus \bar B)}\\
      &\le\abs{S_h q(x_i) -S q(x_i)}\norm{z_{h,i}}_{L^2(\Om)}\\
      &\leq Ch^2 \abs{\ln h}^2\norm{q}_{L^\infty(\Om)}.
    \end{split}
  \end{equation}

  For the second term of~\eqref{eq:sum}, we split
  \begin{equation}\label{eq:1}
    \begin{split}
      (S q (x_i) -\xi_i)z_i - (S_h q (x_i) -\xi_i)z_{h,i}=(S q
      (x_i)-\xi_i)(z_i-z_{h,i})+(S q (x_i)-S_h q (x_i))z_{h,i}
    \end{split}
  \end{equation}
  and obtain for the first term by~\eqref{ellipticregularity} and
  Lemma~\ref{lemma:convergencewithoutsingularity} applied to $z_i$ and $z_{h,i}$ separately on $B_i$
  that
  \[
    \begin{split}
      \norm{(S q (x_i)-\xi_i)(z_i-z_{h,i})}_{L^2(\Om\setminus \bar B)}
      &\le\bigl\{\norm{q}_{L^2(\Om)}+\abs{\xi}\bigr\}\norm{z_i-z_{h,i}}_{L^2(\Om\setminus \bar B)}\\
      &\le Ch^2\lh \bigl\{\norm{q}_{L^2(\Om)}+\abs{\xi}\bigr\}.
    \end{split}
  \]
  For the second term of the right-hand side of~\eqref{eq:1}, we can proceed as for~\eqref{eq:2} to
  obtain
  \[
    \norm{(S q (x_i)-S_h q (x_i))z_{h,i}}_{L^2(\Om\setminus \bar B)}
    \le Ch^2\lh^2\norm{q}_{L^\infty(\Om)}.
  \]
  Collecting the terms, we get for the second term on the right-hand side of~\eqref{eq:sum}
  \[
    \norm{(S\bar q (x_i) -\xi_i)z_i - (S_h\bar q (x_i) -\xi_i)z_{h,i}}_{L^2(\Om\setminus \bar B)}\le
    Ch^2\lh^2\bigl\{\norm{q}_{L^\infty(\Om)}+\abs{\xi}\bigr\}.
  \]
  By adding up the two estimates for~\eqref{eq:sum}, we arrive at the result.
\end{proof}

Based on the previous lemma, we can improve Lemma~\ref{lemma:j-jh} for the case $d=2$. The
improvement consists in the fact that the $L^\infty$ norm of $\delta q$ can be replaced by its $L^2$
norm provided that $\delta q$ vanishes in the neighborhood of the points $x_i$ for $i\in I\setminus
L$.

\begin{lemma} \label{lemma:j-jh_improved}
  Let $d=2$ and $B_i\subset\Om$ open balls with $x_i\in B_i$ for $i\in I\setminus L$. For
  $B=\bigcup_{i\in I\setminus L}B_i$ and  $q, \delta q \in \Qad$ with $\delta q\bigr\rvert_B = 0$, it holds
  \[
    \abs{j'(q)(\delta q) - j'_h(q)(\delta q)} \leq
    Ch^2\lh^2\norm{\delta q}_{L^2(\Om)}.
  \]
  with a constant $C$ independent of $h$.
\end{lemma}
\begin{proof}
  The proof follows the lines of Lemma~\ref{lemma:j-jh} obtaining
  \[
    \abs{j'(q)(\delta q) - j_h'(q)(\delta q)} \le \abs{(z-\tilde z_h, \delta q )} + \abs{(\tilde z_h
    - z_h,\delta q)}.
  \]
  with $\tilde z_h$ solving~\eqref{eq:tilde_zh}. The second term is treated as in
  Lemma~\ref{lemma:j-jh}.  Using Lemma~\ref{lemma:neu} for the first term then implies the result by
  \[
    \abs{(z-\tilde z_h, \delta q )}\le \norm{z-\tilde z_h}_{L^2(\Om\setminus\bar B)}\norm{\delta
    q}_{L^2(\Om)}
  \]
  since $\delta q\bigr\rvert_B=0$.
\end{proof}

%%%%%%%%%%%%%%%%%%%%%%%%%%%%%%%%%%%%%%%%%%%%%%%%%%%%%%%%%%%%%%
\subsection{Variational Discretization}
%%%%%%%%%%%%%%%%%%%%%%%%%%%%%%%%%%%%%%%%%%%%%%%%%%%%%%%%%%%%%%

The following result improves Theorem~\ref{theorem:result1} in the case $d=2$. Since, the result
relies on Lemma~\ref{lemma:cutoffeq} which is available for $d=2$ only, an extension to $d=3$ is not
directly possible.

\begin{theorem} \label{theorem:result1.1}
  Let $d=2$, $\bar q\in \Qad$ be the solution of the continuous problem
  \eqref{definition:contproblem_red}, and $\bar q_h\in \Qad$ be the solution of the corresponding
  discrete problem \eqref{definition:discreteproblem_red} with $\Qhad=\Qad$. Then, it holds
  \[
    \norm{\bar q - \bar q_h}_{L^2(\Om)}\leq C h^2 \abs{\ln h}^2
  \]
  with a constant $C$ independent of $h$.
\end{theorem}
\begin{proof}
  The assertion is proven as Theorem~\ref{theorem:result1}: Let $B=\bigcup_{i\in I\setminus L} B_i$ for the sets
  $B_i$ given by Lemma~\ref{lemma:equalityqqh}. On $B$, it holds $\delta q=\bar q-\bar q_h=0$.
  Hence, by using Lemma~\ref{lemma:j-jh_improved}, we conclude
  \[
    \alpha \norm{ \bar q-\bar q_h }^2_{L^2(\Omega)}\leq \abs{j'(q)(\delta q) - j'_h(q)(\delta q)}
    \le
    C h^2\abs{\ln h}^2\bigl\{\norm{\bar q}_{L^\infty(\Om)}+\abs{\xi}\bigr\}\norm{\bar q - \bar q_h}_{L^2(\Om)}.
  \]
  Dividing by $\norm{\bar q - \bar q_h}_{L^2(\Om)}$ proves the theorem.
\end{proof}

%%%%%%%%%%%%%%%%%%%%%%%%%%%%%%%%%%%%%%%%%%%%%%%%%%%%%%%%%%%%%%
\subsection{Cellwise Constant Control Discretization with Post Processing}\label{sec:post_p}
%%%%%%%%%%%%%%%%%%%%%%%%%%%%%%%%%%%%%%%%%%%%%%%%%%%%%%%%%%%%%%

We adapt the proof technique from \cite{MR2114385}. We split up the mesh into subsets with respect
to the regularity of $\bar q$.  For a given mesh $\Th$, we define three subsets of cells
\[
  \begin{gathered}
\Th^1 = \Set{ K\in \Th | \bar q\bigr\rvert_{K} = a \text{ or } \bar q\bigr\rvert_{K}=b },\quad
    \Th^2 = \Set{ K\in \Th | a< \bar q \bigr\rvert _{K} <b },\\
    \text{and}\quad\Th^3 = \Th \setminus (\Th^1 \cup \Th^2).
  \end{gathered}
\]
Then, $\Th^3$ denotes the set of cells where $\bar q$ is only Lipschitz continuous in contrast to
$\Th^1\cup \Th^2$ which consists of cells $K$ where $\bar q$ is in $H^2(K)$. 
As in~\cite{MR2114385}, we make the following assumption:
\begin{assumption}\label{ass:1}
  We assume the existence of a constant $C$ independent of $h$, such that for all $h$ sufficiently
  small, it holds
  \[
    \sum_{K\in \Th^3}\abs{K} \leq Ch.
  \]
\end{assumption}
Similar assumptions have been made in the case of cellwise linear discretization or a postprocessing
approach in, e.g.,~\cite{MR2192592,MR2114385}.

\begin{remark}
  By~\eqref{proof:boundednessofprojection2.1} applied to $\bar z$, it seems likely that the active set
  $\set{x\in\Om| \bar q(x)=a\vee \bar q(x)=b}$ has rectifiable boundary. This would imply that the
  boundary is a curve of finite length, i.e., it can be covered by a subset $\Th^3\subset\Th$
  fulfilling Assumption~\ref{ass:1}.
\end{remark}

We denote by $S_K$ the centroid of a cell $K\in \Th$. For $w\in C(\Omega)$, we define the projection
$R_hw\in Q_h^c$ by
\begin{equation}\label{eq:Rh}
R_hw\bigr\rvert_{K}=w(S_K) \quad \forall K \in \Th.
\end{equation}
Further, we set
\begin{equation}\label{eq:rh}
  r_h = R_h\bar q,
\end{equation}
which is well-defined by Lemma~\ref{lemma:boundednessofprojection}, and note that by construction,
it holds $r_h\in\Qhad^c$.

\begin{prop} \label{lemma:integrationresult}
  Let $f\in H^2(K)$ for a given cell $K\in\Th$. Then, it holds
  \[
    \left\lvert\int_K(f(x)-f(S_K))\,dx\right\rvert\leq C h^2
    \abs{K}^{\frac12}\norm{\nabla^2f}_{L^2(K)}.
  \]
\end{prop}
\begin{proof}
  See~\cite[Lemma 3.2]{MR2114385}.
\end{proof}

Similar to~\cite[Lemma 3.3]{MR2114385}, we prove the following result adapted to the problem
considered here:
\begin{lemma} \label{lemma:centroidpointwisestabilitystateequation}
  Let $d=2$, $\bar q\in\Qad$ be the solution of \eqref{definition:contproblem_red}, and
  $r_h\in\Qhad^c$ defined by~\eqref{eq:rh}.  Then, there is $h_0>0$ such that for all $i\in
  I\setminus L$ and all $h\le h_0$, it holds
  \[
    \abs{S_h(\bar q - r_h)(x_i)}\leq Ch^2
  \]
  with a constant $C$ independent of $h$.
\end{lemma}
\begin{proof}
  Let $z_{h,i}\in V_h$ be the solution of~\eqref{eq:zih}. Then, we can write
  \[
    S_h(\bar q -  r_h)(x_i) = (\nabla z_{h,i}, \nabla (S_h\bar q - S_h r_h)) = (z_{h,i},\bar q-r_h).
  \]
For $K\in\Th^1$, it holds $\bar q\bigr\rvert_K\in \set{a,b}$. Hence, we have $r_h=\bar q$ there
and  $(z_{h,i},\bar q-r_h)$ vanishes on $\bigcup\Th^1$. Consequently, we get
\[
  (z_{h,i},\bar q-r_h) = \sum_{K\in\Th^2}\int_K z_{h,i}(\bar q-r_h) \, dx+\sum_{K\in\Th^3}\int_K
  z_{h,i}(\bar q-r_h) \, dx.
\]
Lemma~\ref{lemma:boundednessofprojection} ensures the existence of open sets $B_i\subset\Om$ for
$i\in I\setminus L$ with $x_i\in B_i$ and $\bar q(x)\in \set{a,b}$ for $x\in B_i$. For $h\le h_0$,
the sets $B_i$ can be chosen such that $B_i\subset\bigcup\Th^1$. 

We first consider the integral over $K\in \Th^3$.  Again by
Lemma~\ref{lemma:boundednessofprojection}, we know that $\bar q\in W^{1,\infty}(K)$ which implies for
$x\in K$ that
\[
  \abs{\bar q(x) - r_h(x)}=\abs{\bar q(x)-\bar q(S_K)}\leq C \abs{x-S_K}\norm{\nabla \bar
  q}_{L^\infty(\Om)} \leq Ch \norm{\nabla \bar q}_{L^\infty(\Om)}.
\]
For $B=\bigcup_{i\in I\setminus L}B_i$ it holds $B\subset\bigcup\Th^1$. Hence, we have
$\bigcup\Th^3\subset\Om\setminus \bar B$ and one obtains
\[
  \begin{split}
    \left\lvert\sum_{K\in \Th^3}\int_Kz_{h,i}(\bar q-r_h) \, dx\right\rvert 
    &\leq \sum_{K\in\Th^3}\int_{K}\abs{z_{h,i}(\bar q-r_h)}\, dx\\
    &\leq Ch \norm{z_{h,i}}_{L^{\infty}(\Omega \setminus \bar B)}\norm{\nabla \bar q}_{L^\infty(\Om)} \sum_{K\in\Th^3}  \abs{K} 
    \leq C h^2 \norm{\nabla \bar q}_{L^\infty(\Om)}
  \end{split}
\]
by means of Assumption~\ref{ass:1} and Lemma~\ref{lemma:boundednessdiscrete}.

For a cell $K\in\Th^2$, we have that 
\[
  \int_K z_{h,i}r_h  \, dx = \int_K z_{h,i}\bar q(S_K)  \, dx= \int_K z_{h,i}(S_K)\bar q(S_K)\, dx.
\]
Using this, we obtain by Proposition~\ref{lemma:integrationresult} and the Cauchy–Schwarz inequality that
\[
  \left\lvert\sum_{K\in\Th^2}\int_K z_{h,i}(\bar q - r_h) \, dx\right\rvert\leq Ch^2 
  \left( \sum_{K\in \Th^2}\norm{\nabla^2 (z_{h,i}\bar q )}^2_{L^2(K)}\right)^{\frac{1}{2}}.
\]
On $\bigcup\Th^2$, we have $\bar q = -\alpha^{-1}\bar z$.  By the product rule, we get since
$\Th^2\subset\Om\setminus \bar B$ that
\[
  \left( \sum_{K \in \Th^2}\norm{\nabla^2(z_{h,i}\bar q )}^2_{L^2(K)}\right)^{\frac{1}{2}}
  \le C\norm{z_{h,i}}_{L^\infty(\Om\setminus \bar B)} \norm{\nabla^2\bar z}_{L^2(\Om\setminus\bar B)}
  + C\norm{\nabla z_{h,i}}_{L^2(\Om\setminus \bar B)} \norm{\nabla\bar z}_{L^\infty(\Om\setminus
  \bar B)}.
\]
Using Lemma~\ref{lemma:boundednessdiscrete} for the terms involving $z_{h,i}$ and
Lemma~\ref{lemma:lipschitzatedges} for 
\[
  \bar z=\sum_{i\in I\setminus L} (S\bar q(x_i)-\xi_i)z_i,
\]
we obtain the assertion by collecting the previous estimates. 
\end{proof}

\begin{lemma} \label{lemma:centroidpointwisestabilityadjointequation}
  Let $d=2$ and $B_i\subset\Om$ open balls with $x_i\in B_i$ for $i\in I\setminus L$. For $\hat z_h\in V_h$ given as the solution of
  \begin{equation}\label{eq:hat_zh}
    (\nabla \hat z_h,\nabla \phi_h)=\sum_{i\in I}(S_hr_h(x_i)-\xi_i)\phi_h(x_i)\quad\forall \phi_h\in V_h.
  \end{equation}
  for $r_h$ defined by~\eqref{eq:rh} and $B=\bigcup_{i\in I\setminus L}B_i$, it holds
  \[
    \norm{\bar z - \hat z_h}_{L^2(\Omega \setminus \bar B)}\leq C h^2 \abs{\ln h}^2
  \]
  with a constant $C$ independent of $h$.
\end{lemma}
\begin{proof}
  We split
  \[
    \norm{\bar z - \hat z_h}_{L^2(\Omega \setminus \bar B)}\leq \norm{\bar z - \tilde z_h}_{L^2(\Omega
    \setminus \bar B)}+ \norm{\tilde z_h - \hat z_h}_{L^2(\Omega \setminus \bar B)},
  \]
  where $\tilde z_h\in V_h$ is the solution of~\eqref{eq:tilde_zh}. For the first term,
  Lemma~\ref{lemma:neu} states
  \[
    \norm{\bar z-\tilde z_h}_{L^2(\Om\setminus \bar B)} \leq Ch^2\lh^2.
  \]
  For the second term, we get form Lemma~\ref{lemma:boundl2} that
  \[
    \norm{\tilde z_h - \hat z_h}_{L^2(\Omega \setminus \bar B)}\le \norm{\tilde z_h - \hat
    z_h}_{L^2(\Omega)}\le C \sum_{i\in I}\abs{S_h(\bar q-r_h)(x_i)}.
  \]
  Then, by Lemma~\ref{lemma:centroidpointwisestabilitystateequation}, we get
  \[
    \norm{\tilde z_h - \hat z_h}_{L^2(\Omega \setminus \bar B)}\leq Ch^2.
  \]
  Collecting the estimates yields the assertion.
\end{proof}

\begin{lemma} \label{lemma:convergencecentroidcellwise}
  Let $d=2$, $\bar q_h\in \Qhad^c$ be the solution of \eqref{definition:discreteproblem_red} with
  $\Qhad=\Qhad^c$, and $r_h$ defined by~\eqref{eq:rh}.  Then, there exists $h_0>0$ such that for $h\le h_0$,
  it holds
  \[
    \norm{\bar q_h - r_h}_{L^2(\Om)} \leq C h^2\abs{\ln h}^2
  \]
  with a constant $C$ independent of $h$.
\end{lemma}
\begin{proof}
  As in \cite{MR2114385}, we first derive a variational inequality for $r_h$. 
  To this end let $h_0>0$ be sufficiently small such that the sets $\set{B_{0,i}\subset\Om|i\in I\setminus
  L}$ given by  Lemma~\ref{lemma:equalityqqh}, subsets $\set{B_{1,i}\subset B_{0,i}|i\in I\setminus L}$, and
  a subset of cells $\widetilde{\mathcal T}_{h_0}\subset \mathcal{T}_{h_0}$ fulfill the relation
  \[
    B_1\subset \Om_{h_0}\subset B_0
  \]
  for $B_0=\bigcup_{i\in I\setminus L}B_{0,i}$, $\Om_{h_0}=\bigcup \widetilde{\mathcal
  T}_{h_0}$, and $B_1=\bigcup_{i\in I\setminus L}B_{1,i}$.

  By Lemma~\ref{lemma:lipschitzatedges}, we have that $\bar z\in C(\bar \Omega\setminus B_1)$. We now
  apply the optimality condition~\eqref{theorem:continuousproblem:optcondeq}, which holds true for
  all $\delta q\in \Qad$ and also pointwise on $\bar\Om\setminus B_1$:
  \[
    (\alpha \bar q (x)+\bar z(x)) (\delta q - \bar q(x))\geq 0 \qquad \forall \delta q \in
    [a,b],~x\in \bar \Omega \setminus B_1.
  \]
  We apply this formula for $x=S_K$ with $K\in\Th$ and $K\subset\Om\setminus\Om_{h_0}$ and $\delta
  q=\bar q_h(S_K)$ which gives
  \[
    (\alpha r_h(S_K)+ \bar z(S_K) )(\bar q_h(S_K) - r_h(S_K))\geq 0.
  \]
  Integrating this inequality over $K$ and summing this up over $K\in \Th$ with $K\subset\Om
  \setminus \Om_{h_0}$ yields
  \[
    (\alpha r_h+R_h \bar z , \bar q_h - r_h)_{L^2(\Omega \setminus \Om_{h_0})}\geq 0.
  \]
  Noting that $\bar q_h-r_h=0$ on $B_0$ implies
  \[
    (\alpha r_h+R_h \bar z , \bar q_h - r_h)_{L^2(\Omega \setminus \bar B_0)}\geq 0.
  \]
  By testing the discrete optimality condition~\eqref{theorem:discreteproblem:optcondeq} with
  $\delta q_h=r_h$, we get
  \[
    (\alpha \bar q_h+\bar z_h , r_h-\bar q_h)_{L^2(\Om\setminus \bar B_0)}\geq 0
  \]
  again by using $r_h-\bar q_h=0$ on $B_0$.  Adding the last two inequalities results in the estimate
  \begin{equation}\label{eq:b1}
    \alpha \norm{r_h - \bar q_h}_{L^2(\Omega)}^2 = \alpha \norm{r_h-\bar q_h}^2_{L^2(\Omega \setminus
    \bar B_0)} \leq ( R_h \bar z - \bar z_h, \bar q_h - r_h)_{L^2(\Om\setminus \bar B_0)}.
  \end{equation}
  We split the right-hand side of the above inequality to get
  \begin{multline*}
    ( R_h \bar z - \bar z_h, \bar q_h - r_h)_{L^2(\Om\setminus \bar B_0)}
    =(R_h \bar z - \bar z, \bar q_h - r_h)_{L^2(\Om\setminus \bar B_0)}\\+ (\bar z - \hat z_h, \bar q
    _h - r_h)_{L^2(\Om\setminus \bar B_0)}+(\hat z_h -
    \bar z_h,\bar q_h - r_h)_{L^2(\Om\setminus \bar B_0)},
  \end{multline*}
  where $\hat z_h\in V_h$ solves~\eqref{eq:hat_zh}.  We separately estimate the three terms on the
  right-hand side.  

  Using Proposition~\ref{lemma:integrationresult}, the fact that $\bar q_h=r_h$ on
  $B_0$, and that $\bar q_h$ and $r_h$ are piecewise constant, one arrives for the first term at
  \[
    \begin{split}
      (R_h \bar z - \bar z, \bar q_h - r_h)_{L^2(\Omega\setminus \bar B_0)}
      &\le\sum_{\substack{K\in\Th\\K\subset \Om\setminus \Om_{h_0}}}\abs{\bar
    q_h(S_K)-r_h(S_K)}\left\lvert\int_{K}(R_h \bar z - \bar z)\, dx\right\rvert\\
    &\leq Ch^2\sum_{\substack{K\in\Th\\K\subset \Om\setminus \Om_{h_0}}}\abs{\bar q_h(S_K)-r_h(S_K)}\abs{K}^{\frac12}\norm{\nabla^2\bar z}_{L^2(K)}\\
    &\leq Ch^2\norm{\bar q_h - r_h}_{L^2(\Om)} \norm{\bar z}_{H^2(\Omega \setminus \bar B_1)},
  \end{split}
\]
where $\norm{\bar z}_{H^2(\Omega \setminus \bar B_1)}$ is bounded due to
Lemma~\ref{lemma:lipschitzatedges}.

For the second term, it follows by Lemma~\ref{lemma:centroidpointwisestabilityadjointequation}
that
\[
  (\bar z - \hat z_h, \bar q _h - r_h)_{L^2(\Omega\setminus \bar B_0)} \leq C h^2 \abs{\ln h}^2\norm{\bar q _h - r_h}_{L^2(\Om)}.
\]

For the last term, we have since $\bar q_h -r_h=0$ on $B_0$ that
\[
  (\hat z_h - \bar z_h,\bar q_h - r_h)_{L^2(\Omega\setminus \bar B_0)}=(\nabla(\hat z_h - \bar z_h),\nabla S_h (\bar q_h - r_h))
  =-\sum_{i\in I}\left( S_h(\bar q_h - r_h)(x_i) \right)^2\leq 0.
\]

By using the last three estimates and~\eqref{eq:b1}, we complete the proof.
\end{proof}

Using the previous lemmas, we can conclude this section by formulating the error estimate for the
post-processed control $\hat q_h$ given by
\begin{equation}\label{eq:q_hat}
  \hat q_h=P_{[a,b]}(-\alpha^{-1}\bar z_h),
\end{equation}
where $\bar z_h\in V_h$ is the adjoint state associated to the solution $\bar q_h\in \Qhad^c$ of the discrete
problem~\eqref{definition:discreteproblem_red} with $\Qhad=\Qhad^c$.

\begin{theorem} \label{theorem:result5}
  Let $d=2$, $\bar q\in \Qad$ be the solution of the continuous
  problem~\eqref{definition:contproblem_red}, and $\hat q_h$ given by~\eqref{eq:q_hat}. Furthermore,
  let Assumption~\ref{ass:1} hold. Then, it holds
  \[
    \norm{\bar q - \hat q_h}_{L^2(\Om)}\leq C h^2\abs{\ln h}^2
  \]
  with a constant $C$ independent of $h$.
\end{theorem}

\begin{proof}
  By Lemma~\ref{lemma:equalityqqh} there is $B=\bigcup_{i\in I\setminus L} B_i$ such that $\bar
  q=\hat q_h$ on $B$. Then, the Lipschitz continuity of the projection operator $P_{[a,b]}$
  in $L^2(\Om)$ implies
  \[
    \norm{\bar q - \hat q_h}_{L^2(\Om)}
    =\norm{\bar q - \hat q_h}_{L^2(\Om\setminus B)}
    =\norm{P_{[a,b]}(-\alpha^{-1}\bar z)- P_{[a,b]}(-\alpha^{-1}\bar z_h)}_{L^2(\Om\setminus \bar B)}
    \leq C \norm{\bar z - \bar z_h}_{L^2(\Omega \setminus \bar B)}.
  \]
  Then, we split
  \[
    \norm{\bar z - \bar z_h}_{L^2(\Omega \setminus \bar B)}\leq \norm{\bar z - \hat z_h}_{L^2(\Omega
    \setminus \bar B)} + \norm{\hat z_h-\bar z_h}_{L^2(\Om\setminus \bar B)},
  \]
  where $\hat z_h$ solves~\eqref{eq:hat_zh}. For the first term, we have by Lemma~\ref{lemma:centroidpointwisestabilityadjointequation}
  \[
    \norm{\bar z - \hat z_h}_{L^2(\Omega \setminus \bar B)}\leq C h^2 \abs{\ln h}^2.
  \]
  For the second term, we get by the
  Lemmas~\ref{lemma:stability},~\ref{lemma:convergencecentroidcellwise}, and~\ref{lemma:boundl2}.
  \[
    \begin{split}
      \norm{\bar z_h-\hat z_h}_{L^2(\Om\setminus \bar B)}
      &\le\norm{\bar z_h-\hat z_h}_{L^2(\Om)} \leq C \sum_{i\in I}\abs{S_h(\bar q_h-r_h)(x_i)} \leq C\norm{\bar q_h - r_h}_{L^2(\Om)} \\
      &\leq C h^2\abs{\ln h}^2.
    \end{split}
  \]
  Combining the estimates implies the assertion.
\end{proof}

%%%%%%%%%%%%%%%%%%%%%%%%%%%%%%%%%%%%%%%%%%%%%%%%%%%%%%%%%%%%%%
\section{Numerical Results}\label{sec:num}
%%%%%%%%%%%%%%%%%%%%%%%%%%%%%%%%%%%%%%%%%%%%%%%%%%%%%%%%%%%%%%

We give numerical results to confirm the results of the previous sections. To this end we consider
different sample problems.  The optimal control problems are solved by the optimization library
\textsc{RoDoBo} \cite{RoDoBo} and the finite element toolkit \textsc{Gascoigne} \cite{Gascoigne}.

We consider the optimal control problem \eqref{definition:contproblem1} with the slightly modified
state equation 
\[
  -\Delta u=f+q
\]
with given right-hand side $f$ in a ball $\Omega=B_{0.5}(x_1)\subset \R^d$ with $d \in \set{2,3}$.
The cost functional consist of one point evaluation at  $x_1= (0.5, 0.5)^{T}$ for $d=2$ and
respectively $x_1=(0.5, 0.5, 0.5)^T$ for $d=3$.  The choice of $\Omega$ allows to give an exact
solution of $-\Delta z_1 = \delta_{x_1}$ by
\[
  z_1 =
  \begin{cases}
    \frac{1}{2\pi}\ln\frac1{\abs{x-x_1}}-\frac{\ln2}{2\pi} & \text{for } d=2,\\
    \frac1{4\pi}\frac{1}{\abs{x-x_1}}-\frac1{2\pi} & \text{for } d=3,
  \end{cases}
\]
where $\abs{x}$ denotes the Euclidean norm of $x\in \R^d$. We choose $\bar u(x) =
\cos(\pi\abs{x-x_1})$ and $\xi_1=\bar u(x_1)-1$. Hence, the adjoint solution $\bar z$ is then given
as $\bar z=z_1$. By choosing $\alpha=1$, the optimal control fulfills $\bar q=P_{[a,b]}(-\bar z)$.
The right-hand side $f$ is chosen such that $\bar u$ solves the state equation. Thus, we have $f =
-\Delta \cos(\pi\abs{x}) - q$. 

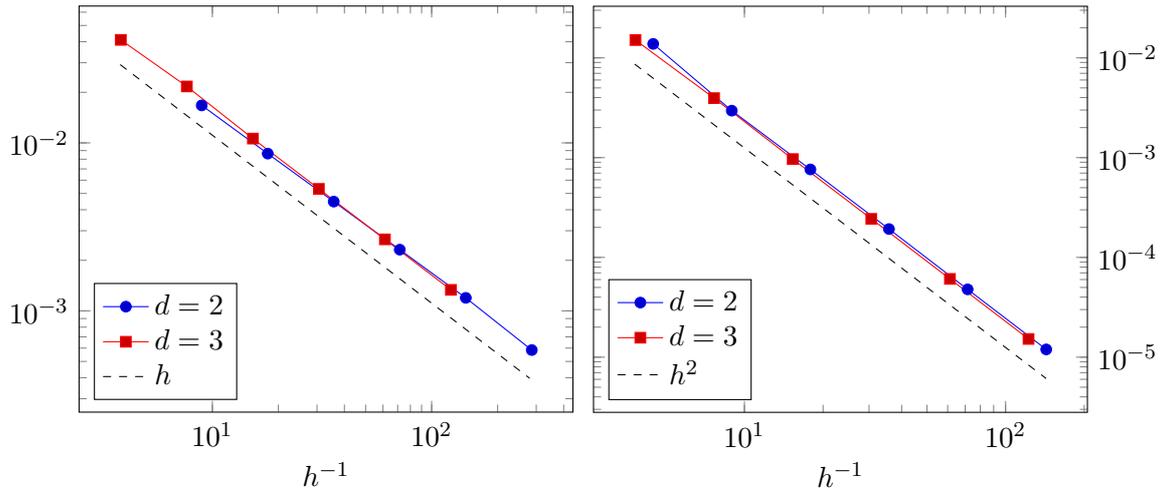
\begin{figure}
  \centering
  \begin{tikzpicture}
    \begin{loglogaxis}
      [
      xlabel=$h^{-1}$,
      legend pos=south west,
      width=0.5\textwidth,
      legend cell align={left},
      xtick={10, 100, 1000, 10000, 100000, 1000000, 10000000, 100000000, 1000000000}
      ]
      \addplot table [y = error, x = dofs]{qerror-c-2.txt};
      \addlegendentry{$d=2$};
      \addplot table [y = error, x = dofs]{qerror-c-3.txt};
      \addlegendentry{$d=3$};
      \addplot [dashed, domain=3.8:280, samples=10,]{1/(9*x)};
      \addlegendentry{$h$};
    \end{loglogaxis}
  \end{tikzpicture}~
  \begin{tikzpicture}
    \begin{loglogaxis}
      [
      xlabel=$h^{-1}$,
      legend pos=south west,
      width=0.5\textwidth,
      ylabel near ticks, yticklabel pos=right,
      legend cell align={left},
      xtick={10, 100, 1000, 10000, 100000, 1000000, 10000000, 100000000, 1000000000}
      ]
      \addplot table [y = error, x = dofs]{qerror-p-2.txt};
      \addlegendentry{$d=2$};
      \addplot table [y = error, x = dofs]{qerror-p-3.txt};
      \addlegendentry{$d=3$};
      \addplot [dashed, domain=3.8:143, samples=10,]{1/(8*x*x)};
      \addlegendentry{$h^2$};
    \end{loglogaxis}
  \end{tikzpicture}
  \caption{Errors $\norm{q-q_h}_{L^2(\Om)}$ for cellwise constant control discretization (left) and
    $\norm{q-\hat q_h}_{L^2(\Om)}$ for cellwise constant control discretization with post processing
  (right)}\label{fig:error}
\end{figure}

First, we present results for cellwise constant discretization of the control. Here, the bounds are
chosen as $-a=b=1$. The numerical results depicted in Figure~\ref{fig:error} (left) confirm the
estimates of Theorem~\ref{theorem:result2}.

Further, for choosing a different value for the bounds, $-a=b=0.2$, we present in
Figure~\ref{fig:error} (right) results for cellwise constant discretization of the control with post
processing. They confirm the estimate of Theorem~\ref{theorem:result5}, which was proved in $d=2$
dimensions only. However, the numerical results for $d=3$ indicate that a similar convergence
result may
also hold in three dimensions.

%%%%%%%%%%%%%%%%%%%%%%%%%%%%%%%%%%%%%%%%%%%%%%%%%%%%%%%%%%%%%%%%%%%%%%%%%%%%%%%
\bibliography{lit}
\bibliographystyle{abbrv}
%%%%%%%%%%%%%%%%%%%%%%%%%%%%%%%%%%%%%%%%%%%%%%%%%%%%%%%%%%%%%%%%%%%%%%%%%%%%%%%

\end{document}